\documentclass[a4paper,10pt]{article}

\usepackage[backref]{hyperref}

\usepackage{color}
\usepackage[utf8]{inputenc}
\usepackage{amsmath}
\usepackage{amsthm}
\usepackage{amsrefs}
\usepackage{amssymb}
\usepackage{mathrsfs}
\usepackage{amscd}
\usepackage{mathtools}
\usepackage{tikz-cd}
\usepackage{verbatim}
\usepackage{slashed}

\newcounter{commentcounter}

\usepackage{color}
\usepackage{ifthen,srcltx}
\newcommand{\showcomments}{yes}

\newsavebox{\commentbox}
\newenvironment{com}%
{\ifthenelse{\equal{\showcomments}{yes}}%
{\footnotemark
        \begin{lrbox}{\commentbox}
        \begin{minipage}[t]{1.25in}\raggedright\sffamily\tiny
        \footnotemark[\arabic{footnote}]}
{\begin{lrbox}{\commentbox}}}
{\ifthenelse{\equal{\showcomments}{yes}}
{\end{minipage}\end{lrbox}\marginpar{\usebox{\commentbox}}}
{\end{lrbox}}}

\title{A Variant of Roe Algebras for Spaces with Cylindrical Ends with Applications in Relative Higher Index Theory}
\author{Mehran Seyedhosseini}
\newtheorem{theorem}{Theorem}[section]
\newtheorem{proposition}[theorem]{Proposition}
\newtheorem{lemma}[theorem]{Lemma}
\theoremstyle{definition}
\newtheorem{definition}[theorem]{Definition}
\theoremstyle{definition}
\newtheorem{remark}[theorem]{Remark}
\newtheorem{corollary}[theorem]{Corollary}
\DeclareMathOperator{\supp}{supp}

\DeclareMathOperator{\prop}{prop}
\DeclareMathOperator{\kernel}{ker}
\DeclareMathOperator{\ind}{ind}
\DeclareMathOperator{\red}{red}

\DeclareMathOperator{\ad}{Ad}
\DeclareMathOperator{\ev}{ev}
\DeclareMathOperator{\Ind}{Ind}

\DeclareMathOperator{\dist}{dist}
\DeclareMathOperator{\iden}{id}
\DeclareMathOperator{\rel}{rel}
\DeclareMathOperator{\st}{st}
\DeclareMathOperator{\Support}{Support}
\DeclareMathOperator{\Id}{Id}
\DeclareMathOperator{\Cl}{Cl}
\DeclareMathOperator{\Spin}{Spin}
\DeclareMathOperator{\Hom}{Hom}

\newcommand{\reals}{\mathbb{R}}

\renewcommand{\restriction}{\mathord{\upharpoonright}}

\date{}
\begin{document}

\maketitle

\begin{abstract}
In this paper we define a variant of Roe algebras for spaces with cylindrical ends and use this to study questions regarding existence and classification of metrics of positive scalar curvature on such manifolds which are collared on the cylindrical end. We discuss how our constructions are related to relative higher index theory as developed by Chang, Weinberger, and Yu and use this relationship to define higher rho-invariants for positive scalar curvature metrics on manifolds with boundary. This paves the way for classification of these metrics. Finally, we use the machinery developed here to give a concise proof of a result of Schick and the author, which relates the relative higher index with indices defined in the presence of positive scalar curvature on the boundary.  
\end{abstract}

\section{Introduction}
The question of whether a given manifold admits a metric of positive scalar curvature has spurred much activity in recent years. One of the main approaches to partially answer this question is index theory. On a closed spin manifold $M$ the Schrödinger-Lichnerowicz formula implies that the nonvanishing of the Fredholm index of the Dirac operator is an obstruction to the existence of positive scalar curvature metric. However, this does not tell the whole story, since there exist spin manifolds with vanishing Fredholm index of the Dirac operator, which however do not admit metrics with positive scalar curvature. One way to obtain more refined invariants from the Dirac operator is to not only consider the dimensions of its kernel and cokernel, but also to consider the action of the fundamental group on them. This gives rise to a higher index for the Dirac operator which is an element of the $K$-theory of the group $C^*$-algebra of the fundamental group. In general, one can associate a class in the $K$-homology of the manifold to the spin Dirac operator and the higher index is obtained as the image of this class under the index map $$
\mu^{\pi_1(M)}:K_*(M) \rightarrow K_*(C^*(\pi_1(M))).$$
The nonvanishing of the higher index gives an obstruction to the existence of positive scalar curvature metrics. In order to prove this one can use the fact that the index map fits in the Higson-Roe exact sequence
$$\hdots \rightarrow S_*^{\pi_1(M)}(M) \rightarrow K_*(M) \rightarrow K_*(C^*(\pi_1(M))) \rightarrow \hdots$$ and that the positivity of the scalar curvature allows the definition of a lift of the fundamental class in $S^{\pi_1(M)}_*(M)$. Given two positive scalar curvature metrics on $M$, one can also define an index difference in $K_{*+1}(C^*(\pi_1(M)))$. These secondary invariants can then also be used for classification of positive scalar curvature metrics up to concordance and bordism. More concretely, in \cite{XY} and \cite {XYZ} the authors use these invariants to prove concrete results on the size of the space of positive scalar metrics on closed manifolds. 

In \cite{CWY} Chang, Weinberger and Yu recently considered the question on compact spin manifolds with boundary. Let $M$ be a compact spin manifold with boundary $N$. They constructed a relative index map 
$$\mu^{\pi_1(M),\pi_1(N)} : K_*(M,N) \rightarrow K_*(C^*(\pi_1(M),\pi_1(N))),$$  where $K_*(M,N)$ and $C^*(\pi_1(M),\pi_1(N))$ denote the relative $K$-homology group and the so-called relative group $C^*$-algebra. One can define a relative class for the Dirac operator on $M$ in the relative $K$-homology group. The relative index is then the image of the latter relative class under the relative index map. Given a positive scalar curvature metric on $M$ which is collared at the boundary, it was shown in \cite{CWY} that the relative index vanishes. A general Riemannian metric which is collared at the boundary and has positive scalar curvature there, also defines an index in $K_*(C^*(\pi_1(M)))$, which vanishes if the metric has positive scalar curvature everywhere. It was shown in \cite{DG} and \cite{SS18} that the latter index maps to the relative index under a certain group homomorphism. Apart from relating previously defined indices to the relative index, this fact also gives a conceptual proof that the relative index is an obstruction to the existence of positive scalar curvature metrics which are collared at the boundary.

The relative index map fits into an exact sequence
$$\hdots \rightarrow S^{\pi_1(M),\pi_1(N)}_*(M,N) \rightarrow K_*(M,N) \rightarrow K_*(C^*(\pi_1(M),\pi_1(N))) \rightarrow \hdots,$$
where $S^{\pi_1(M),\pi_1(N)}_*(M,N)$ is the relative analytic structure group and has different realisations. The main aim of this paper is to answer the following natural question: given a positive scalar curvature metric, which is collared at the boundary, can one define explicitly a secondary invariant in $S^{\pi_1(M),\pi_1(N)}_*(M,N)$ which lifts the relative fundamental class and is useful for classification purposes? Here, note that the exactness of the above sequence immediately implies the existence of some lift or lifts which, however, do not necessarily give us any information about the psc metric at hand. Using the machinery we develop in this paper, we will be able to answer the latter question in the positive. Furthermore, the same machinery allows us to define a higher index difference associated to positive scalar curvature metrics on manifolds with boundary. The definition of such secondary invariants paves the way for generalisations of the known results, such as those of \cite{XY} and \cite {XYZ}, on the size of the space of positive scalar curvature metrics to manifolds with boundary.

Closely related to the question of existence and classification of positive scalar curvature metrics on manifolds with boundary which are collared at the boundary, is the question of existence and classification of positive scalar curvature metrics on manifolds with cylindrical ends, which are collared on the cylindrical end. The usual coarse geometric approach to index theory cannot be applied in this case, since the Roe algebras of spaces with cylindrical ends tend to have vanishing $K$-theory. We deal with this problem by introducing a variant of Roe algebras for such spaces with more interesting $K$-theory. The operators in the new Roe algebras are required to be asymptotically invariant in the cylindrical direction. Such operators can then be evaluated at infinity in a sense to be described later. Let $X$ be a space with cylindrical end and denote by $Y_\infty$ its cylindrical end. Let $\Lambda$ and $\Gamma$ be discrete groups and $\varphi: \Lambda \rightarrow \Gamma$ a group homomorphism. $\varphi$ then induces a map $B\Lambda \rightarrow B\Gamma$ of the classifying spaces of the groups which we can assume to be injective. Given a map $(X,Y_\infty) \rightarrow (B\Gamma,B\Lambda)$ of pairs we construct a long exact sequence
$$\cdots \rightarrow K_*(C^*_{L,0}(\widetilde{X})^{\Gamma,\reals_+,\Lambda}) \rightarrow K_*(C^*_{L}(\widetilde{X})^{\Gamma,\reals_+,\Lambda}) \rightarrow K_*(C^*(\widetilde{X})^{\Gamma,\reals_+,\Lambda}) \rightarrow \cdots.$$
In the above sequence $\widetilde{X}$ denotes the $\Gamma$-cover of $X$ associated to the map $X \rightarrow B\Gamma$ and $C^*(\widetilde{X})^{\Gamma,\reals_+,\Lambda}$ consists, roughly, of operators on $\widetilde{X}$ which are asymptotically invariant and whose evaluation at infinity results in operators admitting $\Lambda$-invariant lifts. For a spin manifold $X$ we associate a fundamental class to the Dirac operator in $K_*(C^*_{L}(\widetilde{X})^{\Gamma,\reals_+,\Lambda})$. The index of the Dirac operator on the manifold with cylindrical end is then defined as the image of the latter class under the map $K_*(C^*_{L}(\widetilde{X})^{\Gamma,\reals_+,\Lambda}) \rightarrow K_*(C^*(\widetilde{X})^{\Gamma,\reals_+,\Lambda})$. Given a positive scalar curvature metric on $X$ which is collared on $Y_\infty$, we define a lift of the fundamental class in $K_*(C^*_{L,0}(\widetilde{X})^{\Gamma,\reals_+,\Lambda})$, which proves that the nonvanishing of the new index is an obstruction to the existence of positive scalar metrics on $X$ and paves the way for classification of such metrics. By removing $Y_\infty$ we obtain a manifold with boundary, which we denote by $\overline{X}$. We prove that there is a commutative diagram of exact sequences 
$$\begin{tikzcd}
   K_*(C^*_{L,0}(\widetilde{X})^{\Gamma,\reals_+,\Lambda}) \arrow{r} \arrow{d} & K_*(C^*_{L}(\widetilde{X})^{\Gamma,\reals_+,\Lambda}) \arrow{r} \arrow{d} & K_*(C^*(\widetilde{X})^{\Gamma,\reals_+,\Lambda}) \arrow{d}\\
   S^{\Gamma,\Lambda}_*(\overline{X},\partial \overline{X}) \arrow{r} & K_*(\overline{X},\partial \overline{X}) \arrow{r} & K_*(C^*(\Gamma,\Lambda)),
\end{tikzcd}$$
where the lower sequence is the relative Higson-Roe sequence mentioned above. Furthermore, we show that the fundamental class of $\widetilde{X}$ maps to the relative fundamental class under the middle vertical map. This shows that the relative index can be obtained from the new index defined in $K_*(C^*(\widetilde{X})^{\Gamma,\reals_+,\Lambda})$ and allows us to define secondary invariants in $S^{\Gamma,\Lambda}_*(\overline{X},\partial\overline{X})$.

As another application of the machinery developed here we give a short proof the main statement of \cite{SS18}.

The paper is organised as follows. The second section is a very short reminder of the picture of $K$-theory for graded $C^*$-algebras due to Trout. In the third section we recall basic notions from coarse geometry and the coarse geometric approach to index theory on manifolds with and without boundary. In the fourth section we introduce variants of Roe algebras for spaces with cylindrical ends and cylinders and define the evaluation at infinity map, which plays an important role in the rest of the paper. In the final sections, we define indices for Dirac operators on manifolds with cylindrical ends and discuss applications to the existence and classification problem for metrics with positive scalar curvature on such manifolds. This is followed by a discussion of the relationship with the relative index for manifolds with boundary and a short proof of a statement on the relationship between the relative index and indices defined in the presence of a positive scalar curvature metric on the boundary. 
\medskip

\noindent\textbf{Acknowledgement}.
I am grateful to Thomas Schick for many inspiring discussions. I would also like to thank Vito Felice Zenobi and the anonymous referee for their comments on earlier drafts of this paper, which helped to improve its presentation.
\section{K-theory for Graded $C^*$-algebras}
In this paper we will use the approach of Trout to $K$-theory of graded $C^*$-algebras. This description of $K$-theory was used by Zeidler in \cite{RZA}, where he proves product formulas for secondary invariants associated to positive scalar curvature metrics. We quickly recall the basics, and refer the reader to \cite{RZA}*{Section 2} for more details.

Let $H$ be a Real $\mathbb{Z}_2$-graded Hilbert space and denote by $\mathbb{K}$ the Real $C^*$-algebra of compact operators on $H$. The $\mathbb{Z}_2$-grading on $H$ induces a $\mathbb{Z}_2$ grading on $\mathbb{K}$ by declaring the even and odd parts to be the set of operators preserving and exchanging the parity of vectors respectively. The Clifford algebra $\Cl_{n,m}$ will be the $C^*$-algebra generated by $\{e_1, \hdots , e_n,\epsilon_1, \hdots , \epsilon_m\}$ subject to the relations $e_ie_j + e_je_i = -2\delta_{ij}, \epsilon_i\epsilon_j + \epsilon_j\epsilon_i = 2\delta_{ij}, e_i\epsilon_j +\epsilon_je_i = 0, e_i^* = -e_i$ and $\epsilon_i^*=\epsilon_i$. The Real structure and the $\mathbb{Z}_2$-grading of $\Cl_{n,m}$ are defined by declaring these generators to be real and odd. Denote by $\mathcal{S}$ the $C^*$-algebra $C_0(\mathbb{R})$ endowed with a Real structure given by complex conjugation and a $\mathbb{Z}_2$-grading defined by declaring the odd and even parts to be the set of odd and even functions. Given Real, $\mathbb{Z}_2$-graded $C^*$-algebras $A$ and $B$ denote by $\Hom(A,B)$ the space of $C^*$-algebra homomorphisms between $A$ and $B$ respecting the Real structures and the $\mathbb{Z}_2$-gradings, by $[A,B]$ the set $\pi_0(\Hom(A,B))$ and by $A\widehat{\otimes} B$ their maximal graded tensor product. The $n$-th $K$-theory group of the Real graded $C^*$-algebra $A$ is defined to be
$$\widehat{K}_n(A) \coloneqq \pi_n(\Hom(\mathcal{S},A\widehat{\otimes}\mathbb{K}))$$
and turns out to be isomorphic to $[\mathcal{S},\Sigma^nA\widehat{\otimes}\mathbb{K}]$, where $\Sigma^nA$ denotes the $n$-th suspension of $A$. Any Real graded homomorphism of $C^*$-algebras $\varphi: \mathcal{S} \rightarrow A$ gives rise to a class $[\varphi] \coloneqq [\varphi \widehat{\otimes} e_{11}] \in \widehat{K}_0(A)$ with $e_{11}$ some rank one projection.

Denote by $\mathcal{S}(-\epsilon,\epsilon)$ the Real graded $C^*$-subalgebra of $\mathcal{S}$ consisting of functions vanishing outside $(-\epsilon,\epsilon)$. For our discussion of secondary invariants we will make use of the fact that the inclusion $\mathcal{S}(-\epsilon,\epsilon) \rightarrow \mathcal{S}$ is a homotopy equivalence. 

\begin{remark}
	From now on, following \cite{RZA}, we use the notation $K_n(A)$ instead of $\widehat{K}_n(A)$ for the $K$-theory of a graded Real $C^\ast$-algebra $A$. Indeed, if we ignore the Real structure and the grading, $\widehat{K}_n(A)$ coincides with the usual complex $K$-theory of the $C^*$-algebra. If there is a Real structure we obtain the the real $K$-theory of the real part of $A$ (see also the appendix of \cite{RZA}).
\end{remark}

\section{Roe Algebras and the Relative Index Map}
In the following, a metric space $X$ is said to have bounded geometry if
\begin{itemize}
	\item there exist a subset $D$ of $X$ and $c >0$ such that any point of $X$ has distance less than $c$ to some point of $D$ and
	\item for any $r>0$ there exists a natural number $N_r$ such that the cardinality of $D \cap U_r(x)$ is less than $N_r$ for any $x \in X$. Here $U_r(x)$ denotes the open $r$-ball with centre $x$.
\end{itemize}
Throughout this section $X$, $Y$ and $Z$ will denote locally compact metric spaces with bounded geometry.
\subsection{Roe Algebras}
Let $\Gamma$ be a discrete group acting freely and properly on $Z$ by isometries. Pulling back functions along the action gives rise to an action $\alpha: \Gamma \rightarrow \text{Aut}(C_0(Z))$. Let $(\rho, U\colon \Gamma \rightarrow U(H))$ be an ample covariant representation of the $C^*$-dynamical system $(C_0(Z),\Gamma,\alpha)$ on a Hilbert space $H$; i.e., $\rho\colon C_0(Z)\rightarrow L(H)$ is a representation of $C_0(Z)$ on $H$, $U\colon \Gamma \rightarrow U(H)$ is a unitary representation of $\Gamma$ on $H$ and the covariance condition
$$\rho(\alpha_\gamma(f)) = U_\gamma\rho(f)U^*_\gamma$$
is satisfied for all $\gamma \in \Gamma$ and $f\in C_0(Z)$. Here ample means that no non-zero element of $C_0(Z)$ acts as a compact operator. The space $H$ will be referred to as a $Z$-module. We will also make use of $\Cl_n$-linear $Z$-modules which are defined analogously by replacing the Hilbert space $H$ with a Real, graded Hilbert $\Cl_n$-module $\mathfrak{H}$ and by requiring the representation $\rho$ to be by adjointable operators. In the following we will denote $\rho(f)$ simply by $f$.
\begin{definition}
 An operator $T \in L(H)$ is called \textit{locally compact} if for all $f \in C_0(Z)$ both $Tf$ and $fT$ are compact. $T$ is called a \textit{finite propagation operator} if there exists $R > 0$ with the property that $fTg=0$ for all $f,g \in C_0(Z)$ with $\dist(\supp f,\supp g) > R$. The smallest such $R$ is called the \textit{propagation} of $T$ and is denoted by $\prop T$. $T$ is called $\Gamma$-equivariant if $T=U^*_\gamma TU_\gamma$ for all $\gamma \in \Gamma$. Similarly, one defines the notions of local compactness and finite propagation for adjointable operators on $\mathfrak{H}$.
\end{definition}

\begin{definition}
 The \textit{equivariant algebraic Roe algebra} is the $*-algebra$ of locally compact, finite propagation, $\Gamma$-equivariant operators on $H$ and is denoted by $\reals(Z)_{\rho}^\Gamma$.
 The \textit{equivariant Roe algebra} is a $C^*$-completion of the algebraic Roe algebra and is denoted by $C^*_{(d)}(Z)_{\rho}^\Gamma$. Here $(d)$ is a placeholder for the chosen completion. Similarly, one defines the \textit{$\Cl_n$-linear equivariant (algebraic) Roe algebra} by using finite propagation, locally compact and equivariant operators on $\mathfrak{H}$. These algebras will be denoted by $\reals(Z;\Cl_n)_{\rho}^\Gamma$ and $C^*_{(d)}(Z;\Cl_n)_{\rho}^\Gamma$.
\end{definition}
\begin{remark}
It follows from Proposition \ref{prop:funcclasroealg} below that the $K$-theory groups of the Roe algebra are independent of the chosen ample representation. We will therefore drop $\rho$ form the notation.
\end{remark}
\begin{remark}
Examples of possible completions are
 \begin{itemize}
  \item the reduced completion $C^*_{\red}(Z)^\Gamma$; i.e.\ the closure of $\reals(Z)^\Gamma$ in $L(H)$,
  \item the maximal completion $C^*_{\max}(Z)^\Gamma$ obtained by taking the completion using the universal $C^*$-norm and
  \item the quotient completion $C^*_{q}(Z)^\Gamma$ introduced in \cite{SS18}.
 \end{itemize}
 In the following we will denote the Roe algebras obtained by the quotient completion simply by $C^*(Z)^\Gamma$ and $C^*(Z;\Cl_n)^\Gamma$. Most of what will follow will be valid for all of the above completions, however we will state all of our results only for the quotient completion. 
\end{remark} 
Later in the paper we will introduce variants of Roe algebras which are suitable for spaces with cylindrical ends and show that the $K$-theory groups of these algebras define functors on a certain category of spaces. Our proofs of the functoriality of the $K$-theory of the new Roe algebras and their independence from the chosen ample modules makes use of the analogues of these results for the classical Roe algebras. Hence, we quickly recall the latter results in the following. Analogues of the results mentioned below hold for the $\Cl_n$-linear versions of the algebras introduced and we will later make use of them.
\begin{definition}[See \cite{RI}*{Chapter 2}]
Let $X$ and $Y$ be locally compact separable proper metric spaces endowed with a free and proper action of a discrete group $\Gamma$ by isometries. A map $f:X \rightarrow Y$ is called coarse if the inverse image of each bounded set of $Y$ under $f$ is bounded and for each $R > 0$ there exists $S > 0$ such that $d_X(x,x^\prime) < R$ implies $d_Y(f(x),f(x^\prime)) < S$.
\label{def:coarsemap}
\end{definition}
\begin{definition}
Let $X$ and $Y$ be as in Definition \ref{def:coarsemap}. Let $H$ and $H^\prime$ denote an $X$ and $Y$-module respectively. The \textit{support} of an operator $T:H \rightarrow H^\prime$ is the complement of the union of all sets $V\times U \subset Y \times X$ with the property that $fTg = 0$ for all $f\in C_0(V)$ and $g \in C_0(U)$. It will be denoted by $\Support(T)$.
\end{definition}
\begin{definition}
Let $X$ and $Y$ be as in Definition \ref{def:coarsemap}. Let $f:X \rightarrow Y$ be a coarse map. Let $H$ and $H^\prime$ denote an $X$ and $Y$-module respectively. An isometry $V:H \rightarrow H^\prime$ is said to \textit{cover} $f$ if there exists an $R > 0$ such that $d_Y(f(x),y) < R$ for all $(y,x) \in \Support(T)$.
\label{def:covisom}
\end{definition}
\begin{lemma}[\cite{HR}*{Lemma 6.3.11}]
Let $f,X,Y,H$ and $H^\prime$ be as in Definition \ref{def:covisom}. If an isometry $V$ covers $f$, then $T \mapsto VTV^*$ defines a map from $\reals(X)^\Gamma$ to $\reals(Y)^\Gamma$, which extends to a map $C^*(X)^\Gamma \rightarrow C^*(Y)^\Gamma$.
\end{lemma}
\begin{proposition}[\cite{HR}*{Proposition 6.3.12}]
Let $f,X,Y,H$ and $H^\prime$ be as in Definition \ref{def:covisom}. There exists an isometry which covers $f$ and thus induces a map $K_*(C^*(X)^\Gamma) \rightarrow K_*(C^*(Y)^\Gamma)$. The latter map is independent of the choice of the isometry covering $f$. In particular, the group $K_*(C^*(X)^\Gamma)$ is independent of the choice of the $X$-module up to a canonical isomorphism.
\label{prop:funcclasroealg}
\end{proposition}
For the rest of the section we consider a space $Z$ with a chosen $Z$-module $H$. In the case the action of $\Gamma$ on $Z$ is cocompact we have the following
\begin{proposition}
 If the action of $\Gamma$ on $Z$ is cocompact, then $K_*(C^*(Z)^\Gamma) \cong K_*(C^*_{q}(\Gamma))$, where $C^*_{q}(\Gamma)$ is the quotient completion of the group ring of $\Gamma$ as introduced in \cite{SS18}.
\end{proposition}
\begin{proof}
In the proof of \cite{HR}*{Lemma 12.5.3} an isomorphism $\reals(X)^{\Gamma} \cong \mathbb{C}[\Gamma]\odot K(H^\prime)$ is given. Here, $\mathbb{C}[\Gamma]$ denotes the complex group ring of $\Gamma$ and $K(H^\prime)$ denotes the algebra of compact operators on a suitable Hilbert space $H^\prime$. This isomorphism becomes an isometry if the left hand side is endowed with the norm of $C^*(X)^\Gamma$ and the right hand side is endowed with the norm of $C^*_q(\Gamma) \otimes K(H^\prime)$ and thus extends to an isomorphism of the latter two algebras. The claim then follows from the stability of $K$-theory.
\end{proof}
Given a $\Gamma$-invariant subset $S \subset Z$ it will be useful to look at the $*$-algebra of operators in $\reals(Z)^\Gamma$ which are supported near $S$ in the sense of the following
\begin{definition}
  Given a subset $S \subset Z$, $T$ is said to be supported near $S$ if there exists an $R > 0$ with the property that $\supp T \subset U_R(S)\times U_R(S)$. Here $U_R(S)$ denotes the open $R$-neighbourhood of $S$.
\end{definition}
\begin{definition}
Let $S$ be a $\Gamma$-invariant subset of $Z$. The equivariant \textit{algebraic Roe algebra of $S$ relative to $Z$} is the subalgebra of $\reals(Z)^\Gamma$ consisting of operators supported near $S$ and will be denoted by $\reals(S\subset Z)^\Gamma$. The equivariant \textit{Roe algebra of $S$ relative to $Z$} is the closure of $\reals(S\subset Z)^\Gamma$ in $C^*(Z)^\Gamma$ and is denoted by $C^*(S\subset Z)^\Gamma$.
\end{definition}
Since $S$ is itself a $\Gamma$-space, it has its own Roe algebra. This is related to the Roe algebra of $S$ relative to $Z$ by the following
\begin{proposition}[\cite{HRY}*{Section 5, Lemma 1}]\label{prop:relabs}
$K_*(C^*(S)^\Gamma) \cong K_*(C^*(S\subset Z)^\Gamma)$.
\end{proposition}
We will also need the notion of support of a vector in $H$.
\begin{definition}
Let $v \in H$. The support of $v$ is the complement of the union of all open subsets $U$ with the property that $fv = 0$ for all $f\in C_0(U)$.
\label{def:suppofvecs}
\end{definition}
\subsection{Yu's Localisation Algebras}
Given a $C^*$-algebra $A$ denote by $\mathfrak{T}A$ the $C^*$-algebra of all uniformly continuous functions $f: [1,\infty) \rightarrow A$ endowed with the supremum norm.
\begin{definition}
	\label{def:localg}
The \textit{equivariant localisation algebra} of $Z$ is defined to be the $C^*$-subalgebra of $\mathfrak{T}C^*(Z)^\Gamma$ generated by elements $f$ satisfying
\begin{itemize}
\item $\prop f(t) < \infty$ for all $t \in [1,\infty)$
\item $\prop f(t) \xrightarrow[]{t\rightarrow \infty} 0$.
\end{itemize}
It will be denoted by $C^*_{L}(Z)^\Gamma$.
\end{definition}
\begin{remark}
	\label{rem:clmax}
	Replacing the quotient completion of the Roe algebra with the maximal (respectively reduced) completion in Definition \ref{def:localg} we can define the maximal (respectively reduced) version of the equivariant localisation algebra $C^*_{L,\max}(Z)^\Gamma$ (respectively $C^*_{L,\red}(Z)^\Gamma$).
\end{remark}
\begin{remark} The $K$-theory of the localisation algebra, obtained by using any of the discussed completions provides a model for the equivariant locally finite $K$-homology. Yu constructed an isomorphism $\Ind_L:K_*^\Gamma(Z) \rightarrow K_*(C^*_{L,\red}(Z)^\Gamma)$, where $K_*^\Gamma(Z)$ denotes the equivariant $KK$-group $KK_*^\Gamma(C_0(Z),\mathbb{C})$. We refer the reader to \cite{Yu} and \cite{QR} for the proof of the isomorphism in the reduced case. See the proof of \cite{SS18}*{Theorem 2.34} and the rest of the discussion in \cite{SS18}*{Section 2} for a proof of the result for an arbitrary completion.
\end{remark}
\begin{definition}
A $\Gamma$-cover $Z$ of a locally compact metric space $M$ is called \textit{nice} if there exists an $\epsilon > 0$ such that the restriction of $Z$ to every $\epsilon$-ball in $M$ is trivial.
\end{definition}
Note that any cover of a compact metric space is nice. The following example is more important for us. Given a compact Riemannian manifold $M$ with boundary $N$ such that the Riemannian metric is collared near $N$, $M_\infty\coloneqq M \cup_N (N\times \mathbb{R}_+)$ can be made into a Riemannian manifold in a natural way (endowing $N\times \mathbb{R}_+$ with the product metric). Any Galois cover of $M_\infty$ is nice.
\begin{proposition}
Let $Z \rightarrow M$ be a nice $\Gamma$-cover. Then there is an isomorphism $K_*(C^*_{L}(Z)^\Gamma) \cong K_*(C^*_{L}(M))$ induced by lifting operators on $M$ with small propagation to equivariant operators on $Z$. In particular $\Ind_L$ gives rise to an isomorphism $K_*(M) \cong K_*(C^*_{L}(Z)^\Gamma)$.
\label{prop:locind}
\begin{remark}
In the following we will assume all covers to be nice.
\end{remark}
\end{proposition}
Given a $\Gamma$-invariant subset $S$ of $Z$ it will be useful to define the localisation algebra of $S$ relative to $Z$.
\begin{definition}
The \textit{equivariant localisation algebra of $S$ relative to $Z$} is defined as the $C^*$-subalgebra of $C^*_{L}(Z)^\Gamma$ generated by elements $f$ with the property that there exists a continuous function $B:[1,\infty) \rightarrow \reals$ vanishing at infinity such that $\supp f(t) \subset U_{B(t)}(S)\times U_{B(t)}(S)$. It will be denoted by $C^*_{L}(S \subset Z)^\Gamma$.
\end{definition}
\begin{proposition}[\cite{RZA}*{Lemma 3.7}]\label{prop:relabsl}
$K_*(C^*_{L}(S)^\Gamma) \cong K_*(C^*_{L}(S \subset Z)^\Gamma)$
\end{proposition}
\begin{definition}
Consider the evaluation at $1$ map $\ev_1\colon C^*_{L}(Z)^\Gamma \rightarrow C^*(Z)^\Gamma$ sending $f$ to $f(1)$. The \textit{equivariant structure algebra} of $Z$ is the kernel of the homomorphism $\ev_1$; i.e., it is the $C^*$-subalgebra of $C^*_{L}(Z)^\Gamma$ consisting of $C^*(Z)^\Gamma$-valued functions $f$ on $[1,\infty)$ with $f(1) = 0$. It is denoted by $C^*_{L,0}(Z)^\Gamma$.
\end{definition}
Given a $\Gamma$-cover $Z \rightarrow M$ induced by a map $M \rightarrow B\Gamma$, with $M$ compact, the index map $\mu^{\Gamma}: K_*(M) \rightarrow K_*(C^*(\Gamma))$ can be defined by
$$K_*(M) \cong K_*(C^*_L(Z)^\Gamma) \xrightarrow{(\ev_1)_*} K_*(C^*(Z)^{\Gamma}) \cong K_*(C^*_q(\Gamma)).$$ Clearly, it fits into a long exact sequence
$$\hdots \rightarrow S_*^{\Gamma}(M) \rightarrow K_*(M) \rightarrow K_*(C^*(\Gamma)) \rightarrow \hdots,$$
where $S_*^{\Gamma}(M)$ denotes $K_*(C^*_{L,0}(Z)^\Gamma)$ and is called the \textit{analytic structure group}. This long exact sequence is called the \textit{Higson-Roe analytic surgery sequence}. 

\subsubsection{Fundamental Class of Dirac Operators} \label{locind}
Now suppose that $Z$ is an $n$-dimensional spin manifold. We assume that $\Gamma$ acts by spin structure preserving isometries. Denote by $\slashed{\mathfrak{S}} = P_{\Spin}(Z) \times_{\Spin}\Cl_n$ the $\Cl_n$-spinor bundle on $Z$. Recall that the $\Cl_n$-linear Dirac operator $\slashed{D}_Z$ on $Z$ (acting on sections of $\slashed{\mathfrak{S}}$) gives rise to a class in $K_*(Z)^\Gamma$. Under the isomorphism of \ref{prop:locind}, this class corresponds to the class $[\slashed{D}_Z] \in \widehat{K}_0(C^*_L(Z;\Cl_n)^\Gamma) \cong K_n(C^*_L(Z)^\Gamma)$ defined by $\varphi_{\slashed{D}_Z} : \mathcal{S} \rightarrow C^*_L(Z;\Cl_n)^\Gamma$ sending $f\in \mathcal{S}$ to $(t \mapsto f(\frac{1}{t}\slashed{D}_Z)) \in C^*_L(Z;\Cl_n)^\Gamma$.
\subsection{The Relative Index Map} \label{relind}
Let $\Lambda$ and $\Gamma$ be discrete groups and $\varphi: \Lambda \rightarrow \Gamma$ a group homomorphism. The homomorphism $\varphi$ gives rise to a continuous map $B\varphi: B\Lambda \rightarrow B\Gamma$. It also induces a map $\varphi :C^*_{\max}(\Lambda)\rightarrow C^*_{\max}(\Gamma)$. We can and will assume that $B\varphi$ is injective. Given a compact space $X$, a subset $Y \subset X$ and a map $f:(X,Y) \rightarrow (B\Gamma,B\Lambda)$ Chang, Weinberger and Yu (\cite{CWY}) define a relative index map $\mu^{\Gamma,\Lambda}:K_*(X,Y) \rightarrow K_*(C^*_{\max}(\Gamma,\Lambda))$. Here $C^*_{\max}(\Gamma, \Lambda) \coloneqq SC_\varphi$ denotes the suspension of the mapping cone of $\varphi$ and is called the (maximal) relative group $C^*$-algebra. If $X$ is not compact, then their construction gives rise to a relative index map with target the $K$-theory group of a relative Roe algebra. Here, we quickly recall the construction of the relative index map. Denote by $\widetilde{X}$ and $\widetilde{Y}$ the $\Gamma$ and $\Lambda$ coverings of $X$ and $Y$ associated to $f$ and $f|_Y$ respectively. Denote by $Y^\prime$ the restriction of $\widetilde{X}$ to $Y$. Using particular $\widetilde{X},Y^\prime$ and $\widetilde{Y}$-modules Chang, Weinberger and Yu construct a morphism of $C^*$-algebras
$$\psi : C^*_{\max}(\widetilde{Y})^\Lambda \rightarrow C^*_{\max}(Y^\prime)^{\frac{\Lambda}{\kernel \phi}} \hookrightarrow C^*_{\max}(\widetilde{X})^\Gamma.$$
We will later discuss the morphism $\psi$ in more detail. Applying $\psi$ pointwise we obtain a morphism
$$\psi_{L} : C^*_{L,\max}(\widetilde{Y})^\Lambda \rightarrow C^*_{L,\max}(Y^\prime)^{\frac{\Lambda}{\kernel \phi}} \hookrightarrow C^*_{L,\max}(\widetilde{X})^\Gamma.$$
See Remark \ref{rem:clmax} above for the definition of $C^*_{L,\max}$. We denote by $C_{\psi_L}$ the mapping cone of the homomorphism $\psi_L$ and by $SC_{\psi_L}$ its suspension. Analogous to the absolute case, there is a map $\Ind_L^{\rel} :K_*(X,Y) \rightarrow K_*(SC_{\psi_L})$. 
\begin{proposition}
$\Ind_L^{\rel}$ is an isomorphism. If, furthermore, $X$ is compact, then $K_{*}(SC_{\psi}) \cong K_*(C^*_{\max}(\Gamma,\Lambda))$.
\end{proposition}
Evaluation at $1$ gives rise to morphisms $\ev_1: C^*_{L,\max}(\widetilde{Y})^\Lambda \rightarrow C^*_{\max}(\widetilde{Y})^\Lambda$ and $\ev_1: C^*_{L,\max}(\widetilde{X})^\Gamma \rightarrow C^*_{\max}(\widetilde{X})^\Gamma$. The diagram
$$\begin{tikzcd}
    C^*_{L,\max}(\widetilde{Y})^\Lambda \arrow{r}{\ev_1} \arrow{d}{\psi_L} & C^*_{\max}(\widetilde{Y})^\Lambda \arrow{d}{\psi}\\
    C^*_{L,\max}(\widetilde{X})^\Gamma  \arrow{r}{\ev_1} & C^*_{\max}(\widetilde{X})^\Gamma
\end{tikzcd}$$
is commutative. Hence, the evaluation at $1$ maps give rise to a morphism $SC_{\psi_L} \rightarrow SC_{\psi}$, which we also denote by $\ev_1$.
\begin{definition}
The \textit{relative index map} $\mu^{\Gamma,\Lambda}$ is defined to be the composition
$$K_*(X,Y) \xrightarrow{\Ind_L^{\rel}} K_*(SC_{\psi_L}) \xrightarrow{(\ev_1)_*} K_*(SC_{\psi}).$$
\end{definition}
\begin{remark}
If $X$ is compact, the isomorphism $K_{*}(SC_{\psi}) \cong K_*(C^*_{\max}(\Gamma,\Lambda))$ allows us to consider $\mu^{\Gamma,\Lambda}$ as a map with values in the $K$-theory of the relative group $C^*$-algebra.
\end{remark}
\begin{remark}
We note that the map $\psi$ is first defined at the level of algebraic Roe algebras. The latter map is continuous only if the algebraic Roe algebras are endowed with a suitable norm (for example the maximal norm). This is the technical reason for the use of maximal completion by Chang, Weinberger and Yu.
\end{remark}
\begin{remark}
Instead of the maximal completion of the group rings and the Roe algebras, one can consider the quotient completion introduced in \cite{SS18} and obtain a similar relative index map. If the group homomorphism $\varphi: \Lambda \rightarrow \Gamma$ is injective, then one can also use the reduced completion of the group rings and Roe algebras. From now on we will only make use of the quotient completion and will most of the time drop the subscript indicating the chosen completion.
\end{remark}
Analogous to the absolute case the relative index map fits into a long exact sequence. The map $\psi_L$ gives rise, by restriction, to a map $\psi_{L,0}:C^*_{L,0}(\widetilde{Y})^\Lambda \rightarrow C^*_{L,0}(\widetilde{X})^\Gamma$. We have a short exact sequence of $C^*$-algebras
$$0\rightarrow SC_{\psi_{L,0}} \rightarrow SC_{\psi_L} \xrightarrow{\ev_1} SC_{\psi}\rightarrow 0,$$
which gives rise to a long exact sequence of $K$-theory groups
$$\cdots \rightarrow K_*(SC_{\psi_{L,0}}) \rightarrow K_*(X,Y) \xrightarrow{\mu^{\Gamma,\Lambda}} K_*(SC_\psi) \rightarrow \cdots .$$
\begin{remark}
Similarly, one defines maps $C^*_{(L)}(\widetilde{Y};\Cl_n)^{\Lambda} \rightarrow C^*_{(L)}(\widetilde{X};\Cl_n)^{\Gamma}$, which we will also denote by $\psi_{(L)}$.
\end{remark}
\subsubsection{The Relative Index of Dirac Operators on Manifolds with Boundary}
Given a compact spin manifold $M$ with boundary $N$ with a metric on $M$ which is collared at the boundary, consider the manifold $M_{\infty}$ obtained by attaching $N_{\infty} \coloneqq N\times[0,\infty)$ to $M$ along $N$. Extend the metric on $M$ to a metric on $M_{\infty}$ using the product metric on the half-cylinder (the metric on $\reals_+$ is the usual one). Denote by $[\slashed{D}_{M_{\infty}}]$ the fundamental class of the Dirac operator on $M_{\infty}$ in $K_*(M_{\infty})$. Given a map $f:(M,N) \rightarrow (B\Gamma, B\Lambda)$, the construction of the previous section gives rise to a relative index map $\mu^{\Gamma,\Lambda}: K_*(M,N) \rightarrow K_*(C^*(\Gamma, \Lambda))$.
\begin{definition}
The relative index of the Dirac operator on $M$ is defined to be the image of $[\slashed{D}_{M_{\infty}}]$ under the composition
$$K_*(M_{\infty}) \xrightarrow{} K_*(M_{\infty},N_\infty) \xrightarrow{\cong} K_*(M,N) \xrightarrow{\mu^{\Gamma,\Lambda}} K_*(C^*(\Gamma, \Lambda)).$$
where the isomorphism $K_*(M_{\infty},N_\infty) \xrightarrow{\cong} K_*(M,N)$ is given by excision.
\end{definition}
The nonvanishing of the relative index obstructs the existence of positive scalar curvature metrics on $M$.
\begin{proposition}[\cite{CWY}*{Proposition 2.18},\cite{SS18}*{Theorem 5.1},\cite{DG}*{Theorem 4.12}]
If there exists a positive scalar curvature metric on $M$ which is collared at the boundary, then the relative index of the Dirac operator on $M$ vanishes.
\end{proposition}
\section{Coarse Spaces with Cylindrical Ends}
Let $X$ be a locally compact metric space with a free and proper action of a discrete group $\Gamma$ by isometries. For a $\Gamma$-invariant subset $Y$ of $X$ we can endow $Y \times \reals$ with a $\Gamma$-action by setting $\gamma(y,t) = (\gamma y,t)$. 
\begin{definition}
Let $X$ and $Y$ be as above. The space $X$ is said to have a cylindrical end with base $Y$ if there exists a $\Gamma$-equivariant isometry $\iota: Y \times [0,\infty) \rightarrow X$ satisfying
\begin{itemize}
\item $\iota((y,0)) = y$
\item $\lim_{R\rightarrow \infty} \dist(\iota(Y\times [R,\infty)),X - Y_{\infty}) = \infty$.
\end{itemize}
Here $Y_\infty$ denotes $\iota(Y\times[0,\infty))$ and $Y\times [0,\infty)$ is endowed with the product metric.
\label{def:cylend}
\end{definition}
\begin{definition}
 Let $(X,Y,\iota)$ and $(X^\prime,Y^\prime,\iota^\prime)$ be spaces with cylindrical ends. A map $f:X \rightarrow X^\prime$ is called a coarse map of spaces with cylindrical ends if it is a coarse map and satisfies
 \begin{itemize}
 \item $f(X\setminus Y_\infty) \subset X^\prime \setminus Y^\prime_{\infty}$ and
 \item $f(\iota(y,t)) = \iota^\prime(g(y),t)$ with $g\coloneqq f|_Y$.
 \end{itemize}
 \label{def:cylendmaps}
\end{definition}
\subsection{Roe Algebras for Spaces with Cylindrical Ends}
Using the isometry $\iota$ one can define an action of $\reals_+$ on $C_0(Y_\infty)$ by setting $L_s(f)(\iota((y,t))) = f(\iota(y,t-s))$ for $t \geq s$ and $L_s(f)(\iota((y,t))) = 0 $ otherwise. We would like to define a variant of Roe algebras for spaces with cylindrical ends. In order to do this we use modules which are equipped with an action of $\reals_+$ by partial isometries, which is compatible with the action of $\reals_+$ on $C_0(Y_\infty)$. Before making this precise we introduce some notation. Let $H_Y$ be a $Y$-module. The Hilbert space $L^2(\reals_+;H_Y)$ can be endowed with the structure of $Y_\infty$-module in a natural way. On $L^2(\reals_+;H_Y)$ one can define a family of partial isometries $P^{\st}_s$ by $P^{\st}_s(f)(t) = f(t-s)$ for $t \geq s$ and $P^{\st}_s(f)(t) = 0$ otherwise.
\begin{definition}
Let $(X,Y,\iota)$ be a space with cylindrical end. A Hilbert space is called an $X$-module tailored to the end if there is a tuple $(\rho,U, \{P_s\}_{s\in \mathbb{R}_+})$ satisfying the following properties:
\begin{itemize}
 \item $(\rho,U)$ is a covariant ample representation of $C_0(X)$ on $H$.
 \item $P_s$ is a strongly continuous family of partial isometries on $H$ satisfying
\begin{itemize}
\item $P_{-s} = P_s^*$
\item $P_{s}^*P_{s} = \widetilde{\rho}(\chi_{\iota(Y\times[0,\infty))})$ for all $s > 0$
\item $P_{s}P_{s}^* = \widetilde{\rho}(\chi_{\iota(Y\times[s,\infty))})$ for all $s>0$
\item $\rho(f)P_s = P_s\rho(L_s(f))$ for all $f \in C_0(Y_\infty)$.
\end{itemize}
\item For some $Y$-module $H_Y$, there is a $\Gamma$-equivariant unitary: $W:\chi_{Y_\infty}H \rightarrow L^2(\reals_+;H_Y)$ which covers the identity and satisfies $WP_s = P^{\st}_sW$. 
\end{itemize}
Here, the tuple $(\rho,U,\{P_s\})$ is part of the structure of the $X$-module and $\widetilde{\rho}$ is the extension of the representation $\rho$ to the bounded Borel functions.
\label{def:xmodttte}
\end{definition}
Similarly, one can define $\Cl_n$-linear modules tailored to the end. The following definitions generalise in an obvious manner to the $\Cl_n$-linear context.
In the rest of the section $(X,Y,\iota)$ will be a space with cylindrical end (endowed with a $\Gamma$-action) and $H$ will denote an $X$-module tailored to the end. We will construct a variant of Roe algebras for spaces with cylindrical ends. Since $H$ is in particular an $X$-module, it can be used to construct the usual equivariant algebraic Roe algebra $\reals(X)^\Gamma$.
\begin{definition}
An operator $T \in L(H)$ is called \textit{asymptotically $\reals_+$-invariant} if
$$\lim_{R \to \infty} \sup_{s>0}||(P_{-s}TP_s - T)\chi_{\iota(Y\times[R,\infty))}|| = 0.$$
\end{definition}
\begin{lemma}
The set of operators in $\reals(X)^\Gamma$, which are asymptotically $\reals_+$-invariant is a $*$-subalgebra.
\end{lemma}
\begin{proof}
Let $S,T \in \reals(X)^\Gamma$ be asymptotically $\reals_+$-invariant. Set $R_0 \coloneqq \prop T$. In the following $\chi_R$ will denote $\chi_{\iota(Y\times[R,\infty))}$. Since $P_sP_{-s} = \chi_s$ (for all $s>0$) and elements in the image of $TP_s\chi_R$ are supported in $\iota(Y\times[R-R_0+s,\infty))$ we have for $R > R_0$
$$(P_{-s}STP_s)\chi_R =(P_{-s}SP_sP_{-s}TP_s)\chi_R.$$
Furthermore, since elements in the image of $(P_{-s}TP_{s})\chi_R$ are supported in $\iota(Y\times[R-R_0,\infty))$ we have
$$(P_{-s}SP_sP_{-s}TP_s)\chi_R = (P_{-s}SP_s\chi_{R-R_0}P_{-s}TP_s)\chi_R.$$ From the asymptotic $\reals_+$-invariance, it follows that $P_{-s}SP_s\chi_{R-R_0} = S\chi_{R-R_0} + E_{R-R_0,s}(S)$ and $P_{-s}TP_s\chi_R = T\chi_R + E_{R,s}(T)$ with
\[
\lim_{R\to\infty}\sup_{s>0}||E_{R-R_0,s}(S)|| = 0 = \lim_{R\to\infty}\sup_{s>0}||E_{R,s}(T)||
\label{eq:lims}\tag{$*$}.
\] Therefore $(P_{-s}STP_s - ST)\chi_R$ is equal to
$$S\chi_{R-R_0}T\chi_R + S\chi_{R-R_0}E_{R,s}(T) + E_{R-R_0,s}(S)T\chi_R + E_{R-R_0,s}(S)E_{R,s}(T) - ST\chi_R$$
$$=S\chi_{R-R_0}E_{R,s}(T) + E_{R-R_0,s}(S)T\chi_R + E_{R-R_0,s}(S)E_{R,s}(T).$$
The latter equality and \eqref{eq:lims} imply that $ST$ is asymptotically $\reals_+$-invariant.
We now show that $T^*$ is also asymptotically $\reals_+$-invariant. We have
$$(P_{-s}T^*P_s - T^*)\chi_R= (\chi_R(P_{-s}TP_s -T))^*.$$
Furthermore, since the propagation of $T$ is $R_0$ the right hand side is equal to $(\chi_R(P_{-s}TP_s -T)\chi_{R-R_0})^* = (\chi_RE_{R-R_0,s}(T))^*$. This shows that $T^*$ is asymptotically $\reals_+$-invariant. The fact that the set of asymptotically $\reals_+$-invariant operators is closed under addition is clear.
\end{proof}
\begin{definition}
The equivariant algebraic Roe algebra of $X$ tailored to the end is the $*$-subalgebra of $\reals(X)^\Gamma$ consisting of asymptotically $\reals_+$-invariant operators. It will be denoted by $\reals(X)^{\Gamma,\reals_+}$. The equivariant Roe algebra of $X$ tailored to the end is the closure of $\reals(X)^{\Gamma,\reals_+}$ in $C^*_{(d)}(X)^\Gamma$ and will be denoted by $C^*_{(d)}(X)^{\Gamma,\reals_+}$. Similarly, using a $\Cl_n$-module tailored to the end, one defines $\reals(X;\Cl_n)^{\Gamma,\reals_+}$ and $C^*_{(d)}(X;\Cl_n)^{\Gamma,\reals_+}$.
\end{definition}
\begin{remark}
Note that the algebraic and $C^*$-algebraic Roe algebras defined above depend, a priori, on the chosen modules tailored to the end. We will see later, that the $K$-theory groups of the $C^*$-algebras defined using different modules are canonically isomorphic.
\end{remark}
\begin{remark}
 The equivariant Roe algebra of $X$ tailored to the end obtained by using the quotient completion will simply be denoted by $C^*(X)^{\Gamma,\reals_+}$. In the following we will only make use of the quotient completion; however, most of the results are also valid for the reduced and maximal completions. 
\end{remark}
Let $(X^\prime,Y^\prime,\iota^\prime)$ be another space with a cylindrical end and $H^\prime$ an $X^\prime$-module tailored to the end given by the data $(\rho^\prime,U^\prime, \{P^\prime_s\})$.
\begin{definition}
 Let $f: X \rightarrow X^\prime$ be a map of spaces with cylindrical ends. An isometry $V:H \rightarrow H^\prime$ is said to cover $f$ if it covers $f$ in the sense of \cite{HR}*{Definition 6.3.9} and satisfies $VP_s = P^\prime_sV$.
 \label{def:acylcovisom}
 \end{definition}
\begin{lemma}
 Let $f$ and $V$ be as in Definition \ref{def:covisom}. Then $T \rightarrow VTV^*$ defines a map $C^*(X)^{\Gamma,\reals_+} \rightarrow C^*(X^\prime)^{\Gamma,\reals_+}$.
\end{lemma}
\begin{proof}
 The fact that conjugation by $V$ gives a map $C^*(X)^\Gamma \rightarrow C^*(X^\prime)^\Gamma$ is the content of \cite{HR}*{Lemma 6.3.11}. We show that if $T \in \reals(X)^\Gamma$ is asymptotically $\reals_+$-invariant, then so is $VTV^*$. In the following $\widetilde{\rho}$ and $\widetilde{\rho}^\prime$ will denote the extension of $\rho$ and $\rho^\prime$ to the bounded Borel functions on $X$ and $X^\prime$ respectively. Using the fact that $V$ intertwines the families $\{P_s\}$ and $\{P^\prime_s\}$ we get 
 $$(P^\prime_{-s}VTV^*P^\prime_s - VTV^*) \rho^\prime(\chi_R) = V(P_{-s}TP_s -T)V^*\rho^\prime(\chi_R)=$$
 $$V(P_{-s}TP_s -T)V^*P^\prime_sP^\prime_{-s} = V(P_{-s}TP_s -T)P_sP_{-s}V^* = V(P_{-s}TP_s -T)\widetilde{\rho}(\chi_R)V^*,$$
 which proves the claim.
\end{proof}

\begin{proposition}
 Let $f:X \rightarrow X^\prime$ be a map of spaces with cylindrical ends. Then there is an isometry $V: H \rightarrow H^\prime$ which covers $f$. Conjugation by $V$ induces a homomorphism $K_*(C^*(X)^{\Gamma,\reals_+})\rightarrow K_*(C^*(X^\prime)^{\Gamma,\reals_+})$ which does not depend on the choice of the covering isometry $V$. In particular, $K_*(C^*(X)^{\Gamma,\reals_+})$ does not depend on the choice of the $X$-module tailored to the end up to a canonical isomorphism.
 \label{prop:funcroealgttte}
\end{proposition}
\begin{proof}
 We prove the existence of an isometry covering $f$. The proof that the induced map on the $K$-theory groups by conjugation with $V$ does not depend on the choice of $V$ is the same as that of \cite{HR}*{Lemma 5.2.4}. We have $H \cong \chi_{X\setminus Y_\infty}H \oplus \chi_{Y_\infty}H \cong \chi_{X\setminus Y_\infty}H \oplus (H_Y\otimes L^2(\reals_+))$. Similarly $H^\prime \cong \chi_{X^\prime\setminus Y^\prime_\infty}H^\prime \oplus (H^\prime_{Y^\prime}\otimes L^2(\reals_+))$. By Proposition \ref{prop:funcclasroealg}, there are isometries $V_1: \chi_{X\setminus Y_\infty}H \rightarrow \chi_{X\setminus Y_\infty}H^\prime$ and $V_2: H_Y \rightarrow H^\prime_{Y^\prime}$ covering the restrictions of $f$ to $X\setminus Y_\infty$ and $Y$ respectively. We use the above decompositions of $H$ and $H^\prime$ and set $V = V_1\oplus(V_2\otimes \Id)$. Since the isomorphisms $\chi_{Y_\infty}H \cong H_Y\otimes L^2(\reals_+)$ and $\chi_{Y^\prime_\infty}H^\prime \cong H^\prime_{Y^\prime} \otimes L^2(\reals_+)$ cover the identity maps on $Y_\infty$ and $Y^\prime_\infty$ respectively, $V$, seen as an isometry from $H$ to $H^\prime$, covers $f$ in the sense of Definition \ref{def:covisom}. Furthermore, the latter isomorphisms intertwine the families $\{P_s\}$ and $\{P^\prime_s\}$ with the standard families of partial isometries $\{P^{\st}_s\}$ on $H_Y\otimes L^2(\reals_+)$ and $H_{Y^\prime}\otimes L^2(\reals_+)$, which implies that $V$ intertwines $\{P_s\}$ and $\{P^\prime_s\}$. Thus, $V$ covers $f$ in the sense of Definition \ref{def:acylcovisom}.
\end{proof}
One can also define localisation and structure algebras tailored to the end.
\begin{definition}
The \textit{equivariant localisation algebra} of $X$ tailored to the end is defined to be the $C^*$-subalgebra of $\mathfrak{T}C^*(X)^{\Gamma,\reals_+}$ generated by elements $f$ satisfying
\begin{itemize}
\item $\prop f(t) < \infty$ for all $t \in [1,\infty)$
\item $\prop f(t) \xrightarrow[]{t\rightarrow \infty} 0$.
\end{itemize}
It will be denoted by $C^*_{L}(X)^{\Gamma,\reals_+}$. The \textit{equivariant structure algebra} of $X$ is defined to be the subalgebra of $C^*_{L}(X)^{\Gamma,\reals_+}$ generated by $f$ which further satisfy $f(1) = 0$. It will be denoted by $C^*_{L,0}(X)^{\Gamma,\reals_+}$.
\end{definition}
\begin{remark}
One can also prove the existence of families of isometries covering a given map in a suitable sense and inducing maps between localisation and structure algebras tailored to the end. One can then deduce an analogue of Proposition \ref{prop:funcroealgttte} for structure and localisation algebras tailored to the end. These statements can be proved by using the approach of the proof of Proposition \ref{prop:funcroealgttte} and slight modifications of the proofs for the classical structure and localisation algebras.
\end{remark}
\subsection{Roe algebras for Cylinders}
One of our main goals in the following is to evaluate asymptotically $\reals_+$-invariant operators on a space $(X,Y,\iota)$ with cylindrical end and obtain $\reals$-invariant operators on the cylinder over $Y$. In this section we define a Roe algebra for cylinders which will be the target of the aforementioned ``evaluation at infinity map". In the following $Y$ will denote a locally compact separable metric space endowed with a free and proper action of a discrete group $\Gamma$ by isometries. Endow $Y \times \reals$ with the product metric. Furthermore, $L^\prime_s(f)(y,t) = f(y,t-s)$ defines an action of $\reals$ on $C_0(Y\times \reals)$. Let $H_Y$ be a $Y$-module. The space $L^2(\reals,H_Y)$ can then be endowed with the structure of a $Y \times \reals$-module. There is a family $\{Q^{\st}_s\}$ of unitaries on $L^2(\reals,H_Y)$ given by the shift of functions in the $\reals$-direction.
\begin{definition}
A Hilbert space $H$ is called a cylindrical $Y\times \reals$-module if there is a tuple $(\rho,U,\{Q_s\})$ satisfying the following properties:
\begin{itemize}
\item $(\rho,U)$ is a covariant ample representation of $C_0(Y\times \reals)$ on $H$.
\item $\{Q_s\}$ is a strongly continuous group of unitaries commuting with the representation $U$ of $\Gamma$ on $H$ and satisfying $\rho(f)Q_s = Q_s\rho(L^\prime_s(f))$.
\item For some $Y$-module $H_Y$, there is a unitary isomorphism $W: H \rightarrow L^2(\reals,H_Y)$ which covers the identity map of $Y \times \reals$ in the sense of Definition \ref{def:covisom}, intertwines the families $\{Q_s\}$ and $\{Q^{\st}_s\}$ and which does not shift the support of vectors in the $\reals$-direction.
\end{itemize}
\end{definition}
A cylindrical $Y\times \reals$-module is in particular a $Y\times \reals$-module and allows us to define the usual Roe algebras $\reals(Y\times \reals)^\Gamma$ and $C^*(Y \times \reals)^\Gamma$
\begin{definition}
An operator $T \in \reals(Y\times\reals)^\Gamma$ is called \textit{$\reals$-invariant} if
$$Q_{-s}TQ_s - T = 0$$
for all $s \in \reals$. The closure of the $*$-algebra of such elements in $C^*(Y\times \reals)^\Gamma$ will be denoted by $C^*(Y \times \reals)^{\Gamma\times\reals}$. Similarly, using a cylindrical $Y\times \reals$-$\Cl_n$-module, one defines $C^*(Y\times \reals;\Cl_n)^{\Gamma \times\reals}$.
\end{definition}

Now let $Y^\prime$ be another space. Let $f: Y \times \reals \rightarrow Y^\prime \times \reals$ be a coarse map, which is the suspension of a map $g : Y \rightarrow Y^\prime$. Let $H$ and $H^\prime$ be cylindrical $Y\times \reals$ and $Y^\prime \times \reals$-modules respectively.  
A slight modification of the proof of Proposition \ref{prop:funcroealgttte}, proves the following
\begin{proposition}
Let $f,H$ and $H^\prime$ be as above. There exists an isometry $V:H \rightarrow H^\prime$ which covers $f$ in the sense of Definition \ref{def:covisom} and intertwines the families $\{Q_s\}$ and $\{Q^\prime_s\}$. Conjugation by $V$ induces a homomorphism $K_*(C^*(Y\times \reals)^{\Gamma \times \reals}) \rightarrow  K_*(C^*(Y^\prime \times \reals)^{\Gamma \times \reals})$. The latter homomorphism is independent of the choice of the isometry $V$ satisfying the above properties. In particular, $K_*(C^*(Y\times \reals)^{\Gamma \times \reals})$ does not depend on the chosen cylindrical $Y \times \reals$-module.
\label{prop:funccylroealg}
\end{proposition}
\subsection{The Evaluation at Infinity Map}
Let $(X,Y,\iota)$ be a space with cylindrical end on which $\Gamma$ acts as above. Asymptotically $\reals_+$-invariant operators can be ``evaluated at infinity'' in the sense of Propositions \ref{prop:defevatinfty} and \ref{prop:evatinfty} to give $\reals$-invariant operators on $Y\times \reals$. In order to do this we first introduce the notion of $(X,Y,\iota)$ modules, which is given by a pair consisting of an $X$-module tailored to the end and a cylindrical $Y\times \reals$-module which are related in a special way.
\begin{definition}
Let $(X,Y,\iota)$ be a space with cylindrical end. A pair $(H,H^\prime)$ of Hilbert spaces is called a \textit{$(X,Y,\iota)$-module}, if there is a tuple $(\rho,\rho^\prime,U,U^\prime, \{P_s\},\{Q_s\},i)$ satisfying the following properties:
\begin{itemize}
 \item $(\rho,U,\{P_s\})$ and $(\rho^\prime,U^\prime,\{Q_s\})$ endow $H$ and $H^\prime$ with the structure of an $X$-module tailored to the end and a cylindrical $Y\times \reals$-module respectively.
 \item $i$ is a unitary $\chi_{Y_\infty} H \rightarrow \chi_{Y\times \reals_+}H^\prime$ intertwining the $\Gamma$-representations and the representations of $C_0(Y_\infty)$ and $C_0(Y\times \reals_+)$ on $\chi_{Y_\infty} H$ and $\chi_{Y\times \reals_+}H^\prime$ respectively.
\item $Q_s \circ i = i \circ P_s \restriction_{\chi_{Y_\infty} H}$ for all $s > 0$.
\end{itemize}
\label{def:xyiotamod}
\end{definition}
\begin{remark}
Note that $\rho^\prime(f)Q_s = Q_s\rho^\prime(L^\prime_s(f))$ in particular implies that $Q_s$ applied to vectors in $H^\prime$ which are supported in $Y\times [R,\infty)$, results in vectors with support in $Y\times [R+s,\infty)$. This observation will be used in the proof of Proposition \ref{prop:defevatinfty}.
\end{remark}
In the following we will call an element $v\in H^\prime$ compactly supported if its support in the sense of Definition \ref{def:suppofvecs} is a compact subset of $Y\times \reals$. The nondegeneracy of $\rho^\prime$ implies that compactly supported vectors are dense in $H^\prime$.
\begin{proposition}
 For $T\in \mathbb{R}(\widetilde{X})^{\Gamma,\mathbb{R}_+}$ and a compactly supported vector $v \in H^\prime$ the limit $T^\infty v \coloneqq \lim_{s \to \infty} Q_{-s}iTi^*Q_sv$ exists in
 $H^\prime$ and the mapping $v \mapsto T^{\infty}v$ extends to a continuous linear map
 $T^{\infty}$ on $H^\prime$. Furthermore, the operator $T^{\infty}$ defined in this way is an element of $\reals(Y\times \reals)^{\Gamma\times\reals}$.
 \label{prop:defevatinfty}
\end{proposition}
\begin{proof}
In the following $\chi_R$ will denote $\chi_{\iota(Y\times[R,\infty))}$ and will be seen as an operator on $H$. $\chi^\prime_R$ will denote $\chi_{Y\times [R,\infty)}$ and will act as an operator on $H^\prime$.
\begin{itemize}
 \item The limit exists: for $\epsilon > 0$ choose $\widetilde{R}$ such that $\sup_{s>0}||(P_{-s}TP_s - T)\chi_R|| < \epsilon$
  for all $R \geq \widetilde{R}$. Let $\widetilde{s}$ be such that
  $Q_s(v)$ is supported on $Y\times \reals_+$ for all $s \geq \widetilde{s}$.
  Set $s_0 = \widetilde{R}+\widetilde{s}$. Then we have
  $$||Q_{s_0+s}^{-1}iTi^*Q_{s_0+s}v - Q_{s_0}^{-1}iTi^*Q_{s_0}v|| = ||Q_{s_0}^{-1}(Q_s^{-1}iTi^*Q_s - iTi^*)Q_{s_0}v||.$$
  Note that $(Q_s^{-1}iTi^*Q_s - iTi^*)Q_{s_0}v = i(P_{-s}TP_s - T)\chi_{\widetilde{R}}i^*Q_{s_0}v$;
  hence
  $$||Q_{s_0}^{-1}(Q_s^{-1}iTi^*Q_s - iTi*)Q_{s_0}v|| < ||(P_{-s}TP_s - T)\chi_{\widetilde{R}}||||v||,$$
  where we use that $Q_{s_0}$ is a unitary. The latter inequality shows that $\{Q_s^{-1}iTi^*Q_sv\}_{s\geq \widetilde{s}}$ is
  a Cauchy net and thus has a limit.
  \item $T^\infty$ is a bounded operator on $H^\prime$: we clearly have $||T^{\infty}v|| \leq ||T||||v||$ for all compactly supported $v$ which shows that $v \mapsto T^\infty v$ is a bounded operator on the dense subspace of compactly supported vectors in $H^\prime$ and thus extends to a bounded operator on $H^\prime$.
  \item $T^\infty$ is an $\reals$ and $\Gamma$-invariant operator: for $t \in \reals$ we have $$Q_{-t}T^\infty Q_{t} v = Q_{-t}(\lim_{s \to \infty}Q_{-s}iTi^*Q_sQ_tv) = \lim_{s \to \infty}Q_{-s-t}iTi^*Q_{s+t}v=$$
  $$\lim_{s \to \infty}Q_{-s}iTi^*Q_sv = T^\infty v$$
  for all compactly supported $v$. Therefore $Q_{-t}T^\infty Q_{t} = T^\infty$. A similar computation and the fact that the $\reals$-action and the $\Gamma$-action on $H^\prime$ commute proves the $\Gamma$-invariance.
  \item $T^\infty$ is locally compact: We show that for $\psi \in C_c(Y\times \reals)$, $\psi T^\infty$ is compact. The proof of the compactness of $T^\infty\psi$ is similar and even more straightforward. There exists $M > 0$ such that the support of $\psi$ is contained in $Y \times [-M,\infty)$. Set $R_0 \coloneqq \prop T$. If $v$ is compactly supported with support in $Y \times (-\infty,-M-R_0)$, then $\psi Q_{-s}iTi^*Q_sv = 0$.
  We thus have a commutative diagram 
  $$\begin{tikzcd}
    H^\prime \arrow{r}{\psi T^\infty} \arrow{d}{\chi^\prime_{-M-R_0}} & H^\prime \\
    \chi^\prime_{-M-R_0}H^\prime  \arrow{r}{\psi T^\infty} & \chi^\prime_{-M-R_0}H^\prime. \arrow[hook]{u}
\end{tikzcd}$$ Therefore, it suffices to show that the restriction of $\psi T^\infty$ to $\chi^\prime_{-M-R_0}H^\prime$ is compact. First we show that $\{\chi^\prime_{-M-R_0}Q_{-s}iTi^*Q_s\chi^\prime_{-M-R_0}\}_{s\geq M+R_0}$ is a norm convergent net of operators on $\chi^\prime_{-M-R_0}H^\prime$. Set $s_1 \coloneqq \widetilde{R}+ M+R_0$. Then, similar to the above computation, we have
$$||\chi^\prime_{-M-R_0}Q_{s_1+s}^{-1}iTi^*Q_{s_1+s}\chi^\prime_{-M-R_0} - \chi^\prime_{-M-R_0}Q_{s_1}^{-1}iTi^*Q_{s_1}\chi^\prime_{-M-R_0}|| = $$
$$||\chi^\prime_{-M-R_0}Q_{s_1}^{-1}(Q_s^{-1}iTi^*Q_s - iTi^*)Q_{s_1}\chi^\prime_{-M-R_0}||.$$
Furthermore,
$$(Q_s^{-1}iTi^*Q_s - iTi^*)Q_{s_1}\chi^\prime_{-M-R_0} = i(P_{-s}TP_s - T)\chi_{\widetilde{R}}i^*Q_{s_1}\chi^\prime_{-M-R_0},$$
which implies
$$||\chi^\prime_{-M-R_0}Q_{s_1}^{-1}(Q_s^{-1}iTi^*Q_s - iTi^*)Q_{s_1}\chi^\prime_{-M-R_0}|| < \epsilon.$$
Hence, $\{\psi\chi^\prime_{-M-R_0}Q_{-s}iTi^*Q_s\chi^\prime_{-M-R_0}\}_{s\geq M+R_0}$ is norm convergent and converges strongly to $\psi T^\infty$ in $L(\chi^\prime_{-M-R_0}H^\prime)$. Thus $\psi T^\infty$ restricted to $\chi^\prime_{-M-R_0}H^\prime$ is actually the norm limit of 
$$\psi\chi^\prime_{-M-R_0}Q_{-s}iTi^*Q_s\chi^\prime_{-M-R_0} = \chi^\prime_{-M-R_0}Q_{-s}L^\prime_s(\psi)iTi^*Q_s\chi^\prime_{-M-R_0}$$
$$= \chi^\prime_{-M-R_0}Q_{-s}iL_s(\psi)Ti^*Q_s\chi^\prime_{-M-R_0}$$
as $s$ tends to infinity. The compactness of $\psi T^\infty\restriction_{\chi^\prime_{-M-R_0}H^\prime}$ then follows from that of $L_s(\psi)T$.
\end{itemize}
\end{proof}
\begin{proposition}
 The map $\ev_{\infty} : \mathbb{R}(\widetilde{X})^{\Gamma,\mathbb{R}_+} \rightarrow
 \reals(Y \times \reals)^{\Gamma \times \reals}$ given by $T \mapsto T^{\infty}$ is continuous if the domain and target space are endowed with the norms of $C^*(\widetilde{X})^{\Gamma,\mathbb{R}_+}$ and $C^*(Y \times \reals)^{\Gamma \times \reals}$ respectively. Thus it gives rise to a morphism of $C^*$-algebras $\ev_{\infty} : C^*(\widetilde{X})^{\Gamma,\mathbb{R}_+} \rightarrow
 C^*(Y \times \reals)^{\Gamma \times \reals}$.
 \label{prop:evatinfty}
\end{proposition}
\begin{proof}
If we endow $\mathbb{R}(\widetilde{X})^{\Gamma,\mathbb{R}_+}$ with the reduced norm, the continuity of the map $\ev_{\infty} : \mathbb{R}(\widetilde{X})^{\Gamma,\mathbb{R}_+} \rightarrow
 C^*_{\red}(Y \times \reals)^{\Gamma\times \reals}$ follows from the proof of the previous proposition. Indeed, we already saw that this
 map is a contraction. The continuity of this map for the quotient completion follows from its continuity for the reduced completion, the commutativity of the the diagram
 $$
 \begin{tikzcd}
 \reals(X)^{\Gamma,\reals_+} \arrow{r}{\ev_\infty} \arrow{d} & \reals(Y \times \reals)^{\Gamma \times \reals} \arrow{d} \\
 \reals(X/N)^{\Gamma/N,\reals_+} \arrow{r}{\ev_\infty} & \reals(Y/N \times \reals)^{\Gamma/N \times \reals}
 \end{tikzcd}
 $$
 for all normal subgroups $N$ of $\Gamma$ and the definition of the quotient completion in \cite{SS18}*{Section 4}. It remains to show that it is a morphism of $*$-algebras. Let $S$ and $T$ be in $\mathbb{R}(\widetilde{X})^{\Gamma,\mathbb{R}_+}$
 and let $v \in H^\prime$ be compactly supported. We have
 $$(TS)^\infty v = \lim_{s}Q_s^{-1}iTSi^*Q_sv = \lim_{s}Q_s^{-1}iTi^*Q_sQ_s^{-1}iSi^*Q_s v =$$
 $$\lim_{s}Q_s^{-1}iTi^*Q_s(S^\infty v + E(s)) = T^\infty(S^\infty v).$$
 The last equality follows from the fact that $||Q_s^{-1}iTi^*Q_s(E(s))|| \leq ||T||||E(s)||$. The rest is clear.
\end{proof}
Thus an $(X,Y,\iota)$-module allows us to define an evaluation at infinity map. Next we will prove a functoriality result, which in particular shows that the induced map on $K$-theory is independent of the chosen $(X,Y,\iota)$-module. Let $(\widehat{X},\widehat{Y},\widehat{\iota})$ be another space with cylindrical end (and a $\Gamma$-action). Let $(\widehat{H},\widehat{H}^\prime)$ be an $(\widehat{X},\widehat{Y},\widehat{\iota})$-module. Let $f:(X,Y,\iota) \rightarrow (\widehat{X},\widehat{Y},\widehat{\iota})$ be a map of spaces with cylindrical ends. In particular, the suspension of the restriction of $f$ to $Y$ defines a map $Y\times \reals \rightarrow \widehat{Y}\times \reals$.  In this situation we have the following
\begin{proposition}
There are isometries $V:H \rightarrow \widehat{H}$ and $V^\prime : H^\prime \rightarrow \widehat{H^\prime}$ which satisfy the conditions of Definition \ref{def:acylcovisom} and Proposition \ref{prop:funccylroealg} respectively and which make the diagram
$$\begin{tikzcd}
    C^*(X)^{\Gamma,\reals_+} \arrow{r}{\ev_\infty} \arrow{d}{\ad_V} & C^*(Y\times\reals)^{\Gamma \times \reals} \arrow{d}{\ad_{V^\prime}}\\
    C^*(\widehat{X})^{\Gamma,\reals_+}  \arrow{r}{\ev_\infty} & C^*(\widehat{Y}\times \reals)^{\Gamma \times \reals}
\end{tikzcd}$$
commutative.
In particular, the map $(ev_\infty)_*:K_*(C^*(X)^{\Gamma,\reals_+}) \rightarrow K_*(C^*(Y\times\reals)^{\Gamma \times \reals})$ does not depend on the choice of the $(X,Y,\iota)$-module up to the usual canonical isomorphisms. 
\end{proposition}
\begin{proof}
Let $V^\prime: H^\prime \rightarrow \widehat{H^\prime}$ satisfy the conditions of Proposition \ref{prop:funccylroealg} and such that $V^\prime$ and $(V^{\prime})^*$ map vectors which are supported in $Y\times \reals_+$ and $\widehat{Y}\times \reals_+$ to vectors which are supported in $\widehat{Y}\times \reals_+$ and $Y\times \reals_+$ respectively. We have decompositions
$H \cong \chi_{X\setminus Y_\infty}H \oplus \chi_{Y_\infty}H$ and $\widehat{H} = \chi_{\widehat{X}\setminus \widehat{Y}_\infty}\widehat{H} \oplus \chi_{\widehat{Y}_\infty}\widehat{H}$. Using these decompositions we define $V$ to be the isometry
$$
\begin{pmatrix}
V_1 & 0 \\
0 & \widehat{i}^*V^\prime i
\end{pmatrix},
$$
where $V_1 : \chi_{X\setminus Y_\infty}H \rightarrow \chi_{\widehat{X}\setminus \widehat{Y}_\infty}\widehat{H}$ is any isometry covering the restriction of $f$ to $X\setminus Y_\infty$ and $\widehat{i}$ is the unitary from the definition of an $(\widehat{X},\widehat{Y},\widehat{\iota})$-module identifying $\chi_{\widehat{Y}_\infty}\widehat{H}$ and $\chi_{\widehat{Y}\times \reals_+}\widehat{H}^\prime$. Now we show that for $T \in \reals(X)^{\Gamma,\reals_+}$, $(\ad_{V^\prime}\circ \ev_\infty)(T) = (\ev_\infty \circ \ad_V)(T)$. This then finishes the proof of the proposition. Let $v \in \widehat{H}^\prime$ be compactly supported. We have
$$(\ad_{V^\prime}\circ \ev_\infty)(T)v = V^\prime \lim_sQ_{-s}iTi^*Q_s{V^\prime}^*v =   \lim_s\widehat{Q}_{-s}V^\prime iTi^*{V^\prime}^*\widehat{Q}_sv.$$
On the other hand $(\ev_\infty \circ \ad_V)(T)v = \lim_s\widehat{Q}_{-s}\widehat{i}VTV^*\widehat{i}^*\widehat{Q}_sv$. Set $R_0 = \prop T$. For $s$ sufficiently large $\widehat{Q}_sv$ is supported in $Y\times (R_0,\infty)$. Therefore
$$\lim_s\widehat{Q}_{-s}\widehat{i}VTV^*\widehat{i}^*\widehat{Q}_sv = \lim_s\widehat{Q}_{-s}\widehat{i}\widehat{i}^*V^\prime i Ti^*{V^\prime}^*\widehat{i}\widehat{i}^*\widehat{Q}_sv = \lim_s\widehat{Q}_{-s}V^\prime iTi^*{V^\prime}^*\widehat{Q}_sv.$$
Hence, $(\ad_{V^\prime}\circ \ev_\infty)(T) = (\ev_\infty \circ \ad_V)(T)$.
\end{proof}
\subsection{\texorpdfstring{$(\Gamma,\Lambda)$-}-equivariant Roe Algebras}
Let $(X,Y,\iota)$ be a space with cylindrical end. We do not assume the existence of an action of $\Gamma$ on $X$. Let $\Lambda,\Gamma$ and $\varphi$ be as in Section \ref{relind}. Suppose there exists a map of pairs $\eta:(X,Y_{\infty} \coloneqq \iota(Y\times\reals_+))\rightarrow (B\Gamma,B\Lambda)$ satisfying $\eta(\iota((y,t))) = \eta (\iota((y,0)))$ for all $t\in \reals_+$. This allows us to define $\Gamma$-coverings $\widetilde{X},Y_{(\infty)}^\prime$ of $X,Y_{(\infty)}$ and a $\Lambda$-covering $\widetilde{Y}_{(\infty)}$ of $Y_{(\infty)}$. We obtain in this way new spaces with cylindrical ends $(\widetilde{X},Y^\prime,\iota^\prime)$ and $(\widetilde{Y}_\infty,\widetilde{Y},\widetilde{\iota})$. In this section the Roe algebras will be constructed using fixed $(\widetilde{X},Y^\prime,\iota^\prime)$ and $(\widetilde {Y}_\infty,\widetilde{Y},\widetilde{\iota})$-modules. The construction of the previous section gives rise to evaluation at infinity maps $C^*(\widetilde{Y_\infty})^{\Lambda,\reals_+} \rightarrow C^*(\widetilde{Y}\times \reals)^{\Lambda \times \reals}$ and $C^*(\widetilde{X})^{\Gamma, \reals_+} \rightarrow C^*(Y^\prime \times \reals)^{\Gamma \times \reals}$. As seen in Section \ref{relind},
Chang, Weinberger and Yu constructed a map $C^*(\widetilde{Y}\times \reals)^\Lambda \rightarrow C^*(Y^\prime \times \reals)^\Gamma$ \footnote{They constructed the map between the maximal Roe algebras. In \cite{SS18} the quotient completion was introduced and it was shown, that one has a similar map between the quotient completions of the equivariant algebraic Roe algebras.}. It is easy to see that this map respects the $\reals$-invariance and asymptotic $\reals_+$-invariance of operators. Thus we get, by restriction, a map $ C^*(\widetilde{Y}\times \reals)^{\Lambda\times\reals} \rightarrow C^*(Y^\prime \times \reals)^{\Gamma\times\reals}$. We abuse the notation and denote all such ``change of group maps" by $\psi$. If necessary, the domain and range will be specified to avoid confusion. The corresponding maps at the levels of the localisation and structure algebras will be denoted by $\psi_L$ and $\psi_{L,0}$ respectively.
\begin{definition}
$T \in C^*(\widetilde{X})^{\Gamma,\reals_+}$ is called \textit{asymptotically $\Lambda$-invariant} if $\ev_{\infty}(T)$ is contained in the image of $\psi$. The pullback of $C^*(\widetilde{X})^{\Gamma,\reals_+}$ and $C^*(\widetilde{Y}\times\reals)^{\Lambda\times\reals}$ along $\ev_\infty$ and $\psi$, is called the \textit{$(\Gamma,\Lambda)$-equivariant Roe algebra of $X$} and will be denoted by $C^*(\widetilde{X})^{\Gamma,\reals_+,\Lambda}$.
\end{definition}
\begin{remark}
	\label{rem:pullback}
By definition, we have a commutative diagram
$$\begin{tikzcd}
    C^*(\widetilde{X})^{\Gamma,\reals_+,\Lambda} \arrow{r} \arrow{d} & C^*(\widetilde{Y} \times \reals)^{\Lambda\times\reals} \arrow{d}{\psi}\\
    C^*(\widetilde{X})^{\Gamma,\reals_+} \arrow{r}{\ev_{\infty}} & C^*(Y^\prime \times \reals)^{\Gamma\times\reals}
\end{tikzcd}$$
and elements of $C^*(\widetilde{X})^{\Gamma,\reals_+,\Lambda}$ are given by pairs $(S,T)$ with $S\in C^*(\widetilde{X})^{\Gamma,\reals_+}, T \in C^*(\widetilde{Y} \times \reals)^{\Lambda\times\reals}$ with $ev_\infty(S) = \psi(T)$.
\end{remark}
\begin{definition}
The \textit{$(\Gamma,\Lambda)$-equivariant localisation algebra} (respectively the \textit{$(\Gamma,\Lambda)$-equivariant structure algebra}) of $X$ is defined to be the pullback of the following diagram
$$
\begin{tikzcd}
     & C^*_{L,(0)}(\widetilde{Y} \times \reals)^{\Lambda\times\reals} \arrow{d}{\psi}\\
    C^*_{L,(0)}(\widetilde{X})^{\Gamma,\reals_+} \arrow{r}{\ev_{\infty}} & C^*_{L,(0)}(Y^\prime \times \reals)^{\Gamma\times\reals}
\end{tikzcd}
$$
It will be denoted by $C^*_{L,(0)}(\widetilde{X})^{\Gamma,\reals_+,\Lambda}$.
\end{definition}
We obtain an analogue of the Higson-Roe sequence for spaces with cylindrical ends: the short exact sequence
$$0 \rightarrow C^*_{L,0}(\widetilde{X})^{\Gamma,\reals_+,\Lambda} \rightarrow C^*_{L}(\widetilde{X})^{\Gamma,\reals_+,\Lambda} \rightarrow C^*(\widetilde{X})^{\Gamma,\reals_+,\Lambda} \rightarrow 0$$
gives rise to a long exact sequence
$$\cdots \rightarrow K_*(C^*_{L,0}(\widetilde{X})^{\Gamma,\reals_+,\Lambda}) \rightarrow K_*(C^*_{L}(\widetilde{X})^{\Gamma,\reals_+,\Lambda}) \rightarrow K_*(C^*(\widetilde{X})^{\Gamma,\reals_+,\Lambda}) \rightarrow \cdots.$$
\section{Index of Dirac Operators on Manifolds with Cylindrical Ends}
Let $X$ be an $n$-dimensional spin manifold with a cylindrical end with base $Y$. By this we mean that $(X,Y,\iota)$ is a space with cylindrical end, $\iota$ is smooth and $X \setminus \iota(Y\times(0,\infty))$ is a smooth codimension zero submanifold with boundary $Y$. We fix a map $\eta:(X,Y_{\infty} \coloneqq \iota(Y\times\reals_+))\rightarrow (B\Gamma,B\Lambda)$ satisfying $\eta(\iota((y,t))) = \eta (\iota((y,0)))$ for all $t\in \reals_+$ which gives rise to certain covers of $X$ and $Y$, which we will denote as in the previous section. Denote by $L^2(\slashed{\mathfrak{S}}_{\widetilde{X}}), L^2(\slashed{\mathfrak{S}}_{Y^\prime \times \reals}),L^2(\slashed{\mathfrak{S}}_{\widetilde{Y}_\infty})$ and $L^2(\slashed{\mathfrak{S}}_{\widetilde{Y}\times \reals})$ the square integrable sections of the $\Cl_n$-spinor bundles on $\widetilde{X}, Y^\prime \times \reals, \widetilde{Y}_\infty$ and $\widetilde{Y}\times \reals$ respectively. The pairs $(L^2(\slashed{\mathfrak{S}}_{\widetilde{X}}),L^2(\slashed{\mathfrak{S}}_{Y^\prime \times \reals}))$ and $(L^2(\slashed{\mathfrak{S}}_{\widetilde{Y}_\infty}),L^2(\slashed{\mathfrak{S}}_{\widetilde{Y}\times \reals}))$ can be given the structure of an $(\widetilde{X},Y^\prime,\iota^\prime) \Cl_n$-module and an $(\widetilde{Y}_\infty,\widetilde{Y},\widetilde{\iota}) \Cl_n$-module in the natural way respectively. In particular, the families of unitaries on $L^2(\slashed{\mathfrak{S}}_{Y^\prime \times \reals})$ and $L^2(\slashed{\mathfrak{S}}_{\widetilde{Y}\times \reals})$ needed in the definition of cylindrical $Y^\prime \times \reals$ and $\widetilde{Y}\times \reals$-modules will be given by the shift of sections in the $\reals$-direction and will be denoted by $\{Q^\prime_s\}$ and $\{\widetilde{Q}_s\}$ respectively. We will use these modules to construct the relevant $C^*$-algebras in the following section. As in Section \ref{locind}, we obtain classes $[\slashed{D}_{\widetilde{X}}]$ and $[\slashed{D}_{\widetilde{Y}\times \reals}]$ in $\widehat{K}_0(C^*_{L}(\widetilde{X};\Cl_n)^\Gamma)$ and $\widehat{K}_0(C^*_{L}(\widetilde{Y}\times \reals;\Cl_n)^\Lambda)$ respectively.
Note that $\widetilde{Y}\times \reals$ is a manifold with cylindrical end with base $\widetilde{Y}$. In the following we will define a fundamental class for the Dirac operators on $X$ and its cylindrical end in the $K$-theory groups of the $(\Gamma,\Lambda)$-equivariant localisation algebra and discuss indices and secondary invariants obtained from it. We will need the following
\begin{lemma}
The following diagrams are commutative
$$\begin{tikzcd}
\mathcal{S} \arrow{rd} \arrow{r} & C^*(\widetilde{X};\Cl_n)^{\Gamma,\reals_+} \arrow{d}{\ev_\infty}\\
   & C^*(Y^\prime \times \reals;\Cl_n)^{\Gamma\times\reals}
\end{tikzcd}
\begin{tikzcd}
\mathcal{S} \arrow{rd} \arrow{r} & C^*(\widetilde{Y}\times \reals;\Cl_n)^{\Lambda \times\reals} \arrow{d}{\psi}\\
   & C^*(Y^\prime \times \reals;\Cl_n)^{\Gamma\times\reals}.
\end{tikzcd}$$
Here $\mathcal{S} \rightarrow C^*(\widetilde{X};\Cl_n)^{\Gamma,\reals_+}$, $\mathcal{S} \rightarrow C^*(Y^\prime \times \reals;\Cl_n)^{\Gamma\times\reals}$ and $\mathcal{S} \rightarrow C^*(\widetilde{Y}\times \reals;\Cl_n)^{\Lambda\times\reals}$ denote the functional calculi for $\slashed{D}_{\widetilde{X}}$, $\slashed{D}_{Y^\prime \times \reals}$ and $\slashed{D}_{\widetilde{Y}\times \reals}$ respectively.
\label{lemma:funccalcphievinfty}
\end{lemma}
\begin{proof}
First note that the isometry $\iota^\prime$ allows us to identify the $\Cl_n$-spinor bundles over $Y^\prime \times \reals_+$ and $Y^\prime_\infty$, which in turn gives rise to the unitary $i^\prime:\chi_{Y^\prime_\infty}L^2(\slashed{\mathfrak{S}}_{\widetilde{X}}) \rightarrow \chi_{Y \times \reals_+}L^2(\slashed{\mathfrak{S}}_{Y^\prime \times \reals})$. Let $v \in L^2(\slashed{\mathfrak{S}}_{Y^\prime \times \reals})$ be compactly supported. For $f \in \mathcal{S}$ whose Fourier transform is supported in $(-r,r)$, it is well known that $f(\slashed{D}_{\widetilde{X}})$ and $f(\slashed{D}_{Y^\prime \times \reals})$ have propagation less than $r$ and depend on the $r$-local geometry in the sense that $f(\slashed{D}_{\widetilde{X}})w$ and $f(\slashed{D}_{Y^\prime \times \reals})v$ depend only on the Riemannian metric in the $r$-neighbourhood of the supports of $w$ and $v$ respectively. For $v \in L^2(\slashed{\mathfrak{S}}_{Y^\prime \times \reals})$ with compact support pick $s_0$ such that $Q^\prime_sv$ is supported in $Y^\prime \times [2r,\infty)$ for all $s>s_0$. The previous observation then implies that $if(\slashed{D}_{\widetilde{X}})i^*Q^\prime_sv = f(\slashed{D}_{Y^\prime \times \reals})Q^\prime_sv$ for all $s>s_0$. Hence
$$\lim_sQ^\prime_{-s}if(\slashed{D}_{\widetilde{X}})i^*Q^\prime_sv = \lim_sQ^\prime_{-s}f(\slashed{D}_{Y^\prime \times \reals})Q^\prime_sv.$$
However, because the Riemannian metric on $Y^\prime \times \reals$ is $\reals$-invariant, $Q^\prime_s$ commutes with the Dirac operator and its functions. This implies that $Q^\prime_{-s}f(\slashed{D}_{Y^\prime \times \reals})Q^\prime_sv = f(\slashed{D}_{Y^\prime \times \reals})v$ and shows that for $f$ with compactly supported Fourier transform $\ev_\infty(f(\slashed{D}_{\widetilde{X}})) = f(\slashed{D}_{Y^\prime \times \reals})$. The commutativity of the left diagram then follows from the fact that the functions in $\mathcal{S}$ with compactly supported Fourier transform form a dense subset.

Now we show the commutativity of the right diagram. First we need to recall one of the main properties of the map $\psi:C^*(\widetilde{Y}\times \reals;\Cl_n)^{\Lambda} \rightarrow C^*(Y^\prime \times \reals;\Cl_n)^\Gamma$.
Since all the covers are assumed to be nice one has bijections $C^*(\widetilde{Y}\times \reals;\Cl_n)^{\Lambda \times \reals}_\epsilon \cong C^*(Y \times \reals;\Cl_n)^\reals_\epsilon$ and $C^*(Y^\prime \times \reals;\Cl_n)^{\Gamma \times \reals}_\epsilon \cong C^*(Y\times \reals;\Cl_n)^\reals_\epsilon$, where $C^*(Y \times \reals;\Cl_n)^\reals$ is constructed using $L^2(\slashed{\mathfrak{S}}_{Y\times \reals})$ as the $Y\times \reals$-module, $\epsilon$ is a sufficiently small positive real number, and $C^*(\cdot)^\cdot_\epsilon$ denotes the set of elements in the corresponding Roe algebra which have propagation less than $\epsilon$. The bijections are given by pushdowns and lifts of operators on different covers. Furthermore, $\psi$ makes the diagram
$$\begin{tikzcd}
C^*(\widetilde{Y}\times \reals;\Cl_n)^{\Lambda \times \reals}_\epsilon \arrow{rd}{\cong} \arrow{r}{\psi} & C^*(Y^\prime \times \reals;\Cl_n)^{\Gamma \times \reals}_\epsilon \arrow{d}{\cong}\\
   & C^*(Y \times \reals;\Cl_n)^{\reals}_\epsilon.
\end{tikzcd}$$
commutative. Let $f\in \mathcal{S}$ have a Fourier transform which is supported in $(-\epsilon^\prime,\epsilon^\prime)$, with $\epsilon^\prime$ sufficiently small. The observation that $f$ applied to the different Dirac operators depends only on the $\epsilon^\prime$-local geometry and the niceness of covers imply that $f(\slashed{D}_{\widetilde{Y}\times \reals}), f(\slashed{D}_{Y \times \reals})$ and $f(\slashed{D}_{Y^\prime \times \reals})$ correspond to each other under the pushdown/lift maps. The commutativity of the latter diagram then implies that for $f$ with the above property $\psi(f(\slashed{D}_{\widetilde{Y}\times \reals})) = f(\slashed{D}_{Y^\prime \times \reals})$. The commutativity of the right diagram in the claim of the lemma then follows from the fact, that the $C^*$-subalgebra of $\mathcal{S}$ generated by functions whose Fourier transform is supported in a fixed interval $(-C,C)$ is the whole of $\mathcal{S}$, since it separates points.
\end{proof}
Lemma \ref{lemma:funccalcphievinfty} allows us to make the following
\begin{definition}
The $(\Gamma,\Lambda)$-fundamental class of $X$ is the class $[\slashed{D}_{\widetilde{X},\widetilde{Y}}] \in \widehat{K}_0(C^*_L(\widetilde{X};\Cl_n)^{\Gamma,\reals_+,\Lambda}) \cong K_n(C^*_L(\widetilde{X})^{\Gamma,\reals_+,\Lambda})$ defined by 
$$\varphi_{\slashed{D}_{\widetilde{X},\widetilde{Y}}}: \mathcal{S} \rightarrow C^*_L(\widetilde{X};\Cl_n)^{\Gamma,\reals_+,\Lambda}, f \mapsto ( t \mapsto (f(\frac{1}{t}\slashed{D}_{\widetilde{X}}),f(\frac{1}{t}\slashed{D}_{\widetilde{Y} \times \reals})).$$ The $(\Gamma, \Lambda)$-index of the Dirac operator associated to the map $\eta:(X,Y_{\infty} \coloneqq \iota(Y\times\reals_+))\rightarrow (B\Gamma,B\Lambda)$ as above is defined to be the image of $[\slashed{D}_{\widetilde{X},\widetilde{Y}}]$ under the map $(\ev_1)_*:K_*(C^*_{L}(\widetilde{X})^{\Gamma,\reals_+,\Lambda}) \rightarrow K_*(C^*(\widetilde{X})^{\Gamma,\reals_+,\Lambda})$.
\end{definition}
\subsection{Application to Existence and Classification of Positive Scalar Curvature Metrics}
Suppose that the scalar curvature of the metric $g$ on $X$ is bounded from below by $\epsilon$. The same then holds for the lifts of $g$ to various covers of $X$ and $Y_{(\infty)}$. This implies that the spectra of the various Dirac operators considered here do not intersect the interval $(-\frac{\sqrt{\epsilon}}{4},\frac{\sqrt{\epsilon}}{4})$. Let $h$ be a homotopy inverse to the inclusion $\mathcal{S}(-\frac{\sqrt{\epsilon}}{4},\frac{\sqrt{\epsilon}}{4}) \rightarrow \mathcal{S}$. 
\begin{definition}
Let $g$ be as above. The $(\Gamma, \Lambda)$-rho-invariant of $g$ is the class in $K_0(C^*_{L,0}(\widetilde{X};\Cl_n)^{\Gamma,\reals_+,\Lambda}) \cong K_n(C^*_{L,0}(\widetilde{X})^{\Gamma,\reals_+,\Lambda})$ defined by the morphism
$$\varphi_{\slashed{D}_{\widetilde{X},\widetilde{Y}}} \circ h: \mathcal{S} \rightarrow C^*_{L,0}(\widetilde{X})^{\Gamma,\reals_+,\Lambda}$$
and will be denoted by $\rho^{\Gamma,\Lambda}(g)$.
\end{definition}
Clearly, $\rho^{\Gamma,\Lambda}(g)$ lifts $[D_{\widetilde{X},\widetilde{Y}}]$ and by the exactness of the sequence
$$\cdots \rightarrow K_*(C^*_{L,0}(\widetilde{X})^{\Gamma,\reals_+,\Lambda}) \rightarrow K_*(C^*_{L}(\widetilde{X})^{\Gamma,\reals_+,\Lambda}) \rightarrow K_*(C^*(\widetilde{X})^{\Gamma,\reals_+,\Lambda}) \rightarrow \cdots$$
we have the following
\begin{proposition}

If the metric on $X$ has positive scalar curvature then the $(\Gamma,\Lambda)$-index of the Dirac operator vanishes.
\label{prop:vanishing}
\end{proposition}
One can define a notion of concordance for positive scalar curvature metrics on manifolds with cylindrical ends. Let $g$ and $g^\prime$ be such metrics on $X$. They are called concordant if there exist a positive scalar curvature metric $G$ on $X \times \reals$ and a map $j:Y\times \reals \times\reals_+ \rightarrow Y_\infty \times \reals$ which makes $(X\times \reals, Y\times \reals, j)$ a manifold with cylindrical end and such that $G$ restricted to $X\times (1,\infty)$ is $g + dt^2$ and restricted to $X \times (-\infty,0)$ is $g^\prime + dt^2$. Using the strategy of Zeidler in \cite{RZA} and by replacing the usual Roe, localisation and structure algebras by their $(\Gamma,\Lambda)$-invariant counterparts one can without much difficulty prove a partitioned manifold index theorem for secondary invariants for manifolds with cylindrical ends and prove the concordance invariance of the $(\Gamma,\Lambda)$-rho-invariant. However we refrain from discussing this, since it does not entail any novelties.

More generally following the approach of \cite{RZA} we define partial $(\Gamma,\Lambda)$-rho-invariants associated to metrics having positive scalar curvature outside of a given subset $Z$ of $X$. Denote by $Z^\prime$ and and $Z^{\prime\prime}$ the preimages of $Z$ and $(Z \cap \iota({Y\times \{1\}})) \times \reals$ under the covering maps $\widetilde{X} \rightarrow X$ and $\widetilde{Y}\times \reals \rightarrow Y \times \reals$ respectively. Denote by $C^*(Z^\prime \subset \widetilde{X})^{\Gamma,\reals_+,\Lambda}$ the $C^*$-subalgebra of $C^*(\widetilde{X})^{\Gamma,\reals_+,\Lambda}$ consisting of elements $(T_1,T_2)$ with $T_1 \in C^*(Z^\prime \subset \widetilde{X})^\Gamma$ and $T_2 \in C^*(Z^{\prime\prime} \subset \widetilde{Y}\times \reals)^\Lambda$. Denote by $C^*_{L,Z^\prime}(\widetilde{X})^{\Gamma,\reals_+,\Lambda}$ the preimage of $C^*(Z^\prime \subset \widetilde{X})^{\Gamma,\reals_+,\Lambda}$ under the evaluation at $1$ map. The justification for the following definition is provided by \cite{RP}*{Lemma 2.3}.
\begin{definition}
Given a metric $g$ on $X$ which is collared at the boundary whose scalar curvature is bounded below by $\epsilon > 0$ outside of a subset $Z$ define the class $\rho^{\Gamma,\Lambda}_{Z}(g)$ by the morphism
$$\varphi_{\slashed{D}_{\widetilde{X},\widetilde{Y}}} \circ h: \mathcal{S} \rightarrow C^*_{L,Z^\prime}(\widetilde{X})^{\Gamma,\reals_+,\Lambda}.$$
\end{definition}

Another higher index theoretic notion which has been successfully used to obtain information about the size of the space of positive scalar curvature metrics on closed manifolds is the higher index difference, which gives rise to a map from the space of positive scalar curvature metrics to the $K$-theory of the group $C^*$-algebra of the manifold. We now show that one can easily define a $(\Gamma,\Lambda)$-index difference of two positive scalar curvature metrics for manifolds with cylindrical ends. This becomes particularly interesting after we discuss the application of the above machinery to relative higher index theory in the next section. Let $g_0$ and $g_1$ be two metrics on $X$ with scalar curvature bounded below by $\epsilon > 0$ which are collared on the cylindrical end. Define a metric $G$ on $X \times \reals$ which restricts to $g_0 \oplus dt^2$ and $g_1 \oplus dt^2$ on $X \times [0,\infty)$ and $X \times (-\infty, -1)$ respectively and which is collared on the cylindrical end in the $X$-direction.
\begin{definition}
Let $g_0,g_1$ and $G$ be as above. The $(\Gamma,\Lambda)$-index difference of $g_0$ and $g_1$ is the image of $\rho^{\Gamma,\Lambda}_{X \times [0,1]}(G)$ under the composition
$$K_{n+1}(C^*_{L,\widetilde{X} \times [0,1]}(\widetilde{X})^{\Gamma,\reals_+,\Lambda}) \xrightarrow{(\ev_1)_*} K_{n+1}(C^*(\widetilde{X} \times [0,1] \subset \widetilde{X} \times \reals)^{\Gamma,\reals_+,\Lambda})$$
$$\rightarrow K_{n+1}(C^*(\widetilde{X})^{\Gamma,\reals_+,\Lambda}),$$
where the last map is induced by projection on $\widetilde{X}$. It will be denoted by $\ind^{\Gamma,\Lambda}(g_0,g_1)$.
\end{definition}

\subsection{Relationship to the Relative Index of Chang, Weinberger and Yu}
As mentioned above, the relative index map of Chang, Weinberger and Yu for manifolds with boundary takes values in mapping cones of equivariant Roe algebras. Note that given a manifold $(X,Y,\iota)$ with cylindrical end, $\overline{X} \coloneqq X \setminus \iota(Y \times (0,\infty))$ is a manifold with boundary $Y$. By restriction, we obtain a map $\eta:(\overline{X},Y) \rightarrow (B\Gamma,B\Lambda)$. To avoid confusion in the following we denote the ``change of group map" 
\begin{equation}\label{eq:beta}C^*_{(L,(0))}(\widetilde{Y})^\Lambda \rightarrow C^*_{(L,(0))}(\widetilde{\overline{X}})^\Gamma \tag{\(\star\)}
\end{equation} introduced in Section \ref{relind} by $\beta_{(L,(0))}$. All the other change of group map will still be denoted by $\psi_{(L,(0))}$. In the following, we will see that there exists a commutative diagram of exact sequences
$$\begin{tikzcd}
   K_*(C^*_{L,0}(\widetilde{X})^{\Gamma,\reals_+,\Lambda}) \arrow{r} \arrow{d} & K_*(C^*_{L}(\widetilde{X})^{\Gamma,\reals_+,\Lambda}) \arrow{r} \arrow{d} & K_*(C^*(\widetilde{X})^{\Gamma,\reals_+,\Lambda}) \arrow{d}\\
   K_*(SC_{\beta_{L,0}}) \arrow{r} & K_*(SC_{\beta_L}) \arrow{r} & K_*(SC_{\beta}).
\end{tikzcd}$$
\begin{remark}\label{rem:phievinfty}
In the following, we will denote by $(\psi, \ev_\infty)$ the homomorphism
$$
C^*(\widetilde{Y}_\infty)^{\Lambda,\reals_+} \rightarrow  C^*(\widetilde{X})^{\Gamma,\reals_+,\Lambda} \qquad T \mapsto (\psi(T),\ev_\infty(T))$$
where we have used the realisation of $C^*(\widetilde{X})^{\Gamma,\reals_+,\Lambda}$ as a subalgebra of $C^*(\widetilde{X})^{\Gamma,\reals_+} \times C^*(\widetilde{Y}\times \reals)^{\Lambda\times\reals}$ (see Remark \ref{rem:pullback}). By pointwise application of $(\psi,\ev_{\infty})$ we also define maps
$$(\psi_{L},{\ev_\infty}_{L})\colon C^*_L(\widetilde{Y}_\infty)^{\Lambda,\reals_+} \rightarrow  C^*_L(\widetilde{X})^{\Gamma,\reals_+,\Lambda}$$ and
$$(\psi_{L},{\ev_\infty}_{L,0})\colon C^*_{L,0}(\widetilde{Y}_\infty)^{\Lambda,\reals_+} \rightarrow  C^*_{L,0}(\widetilde{X})^{\Gamma,\reals_+,\Lambda}$$
Furthermore, the arrows $C^*(\widetilde{X})^{\Gamma,\reals_+,\Lambda} \rightarrow C^*(\widetilde{Y}\times \reals)^{\Lambda\times\reals}$ and $C^*(\widetilde{\overline{X}} \subset \widetilde{X})^\Gamma \rightarrow C^*(\widetilde{X})^{\Gamma,\reals_+,\Lambda}$ will always denote the projection map $(S,T) \rightarrow T$ and the injection map $S \mapsto (S,0)$ respectively (again we refer the reader to Remark \ref{rem:pullback}).
\end{remark}
\begin{proposition}
The following is a commutative diagram of short exact sequences
$$\begin{tikzcd}
   0 \arrow{r} & C^*(\widetilde{Y} \subset \widetilde{Y}_\infty)^\Lambda \arrow{r} \arrow{d}{\psi} & C^*(\widetilde{Y}_\infty)^{\Lambda,\reals_+} \arrow{r}{\ev_{\infty}} \arrow{d}{(\psi,\ev_\infty)} & C^*(\widetilde{Y}\times \reals)^{\Lambda\times\reals} \arrow{r} \arrow{d}{\iden} & 0\\
   0 \arrow{r} & C^*(\widetilde{\overline{X}} \subset \widetilde{X})^\Gamma \arrow{r} & C^*(\widetilde{X})^{\Gamma,\reals_+,\Lambda} \arrow{r} & C^*(\widetilde{Y}\times \reals)^{\Lambda\times\reals} \arrow{r} \arrow{r} & 0.
\end{tikzcd}$$
Analogous diagrams exist when $C^*$ is replaced by $C^*_{L}$ and $C^*_{L,0}$.
\label{prop:excprep}
\end{proposition}
\begin{proof}
We first show that the first row is exact. It follows immediately from the definition of $\ev_\infty$ that $\reals(\widetilde{Y} \subset \widetilde{Y}_\infty)^\Lambda$ is in its kernel. By continuity we get that $C^*(\widetilde{Y} \subset \widetilde{Y}_\infty)$ is in the kernel of $\ev_\infty$. Furthermore, \cite{HankePapeSchick}*{Lemma 3.12} implies that the kernel of $\ev_\infty$ is exactly $C^*(\widetilde{Y}\subset \widetilde{Y}_\infty)^\Lambda$. It remains to show that $\ev_\infty$ is surjective. For $T \in \reals(\widetilde{Y}\times \reals)^{\Lambda \times \reals}$, the operator $\chi_{\widetilde{Y} \times \reals_+}T\chi_{\widetilde{Y} \times \reals_+}$ maps to $T$ under $\ev_\infty$. The surjectivity then follows from the fact that the image of a homomorphism of $C^*$-algebras is closed. The exactness of the second row can be proven using similar arguments. However we note that the exactness in the middle uses the fact that $\lim_{R\rightarrow \infty} \dist(\iota(Y^\prime\times [R,\infty)),\widetilde{X} - Y^\prime_\infty) = \infty$ (see Definition \ref{def:cylend}). The commutativity of the diagram is a direct consequence of the definitions of the involved maps.
\end{proof}
\begin{remark}
	Denote by $C_{C^*(\widetilde{Y} \subset \widetilde{Y_\infty})^{\Lambda} \rightarrow C^*(\widetilde{\overline{X}} \subset \widetilde{X})^\Gamma}$, $C_{(\psi,\ev_{\infty})}$ and $C_{\iden}$ the mapping cones of the homomorphisms $\psi\colon C^*(\widetilde{Y} \subset \widetilde{Y_\infty})^{\Lambda} \rightarrow C^*(\widetilde{\overline{X}} \subset \widetilde{X})^\Gamma$, $(\psi,\ev_\infty)$ and the identity map $\iden\colon C^*(\widetilde{Y}\times\mathbb{R})^{\Lambda \times \mathbb{R}} \rightarrow C^*(\widetilde{Y}\times\mathbb{R})^{\Lambda \times \mathbb{R}}$ respectively. From the commutativity of the diagram of Proposition \ref{prop:excprep} we obtain the exact sequence
	$$0\rightarrow C_{C^*(\widetilde{Y} \subset \widetilde{Y_\infty})^{\Lambda} \rightarrow C^*(\widetilde{\overline{X}} \subset \widetilde{X})^\Gamma} \rightarrow C_{(\psi,\ev_\infty)}\rightarrow C_{\iden} \rightarrow 0.$$
	Analogous exact sequences exist when $C^*$ is replaced by $C^*_{L}$ and $C^*_{L,0}$.
\end{remark}
\begin{proposition}
The inclusion
$$C_{C^*(\widetilde{Y} \subset \widetilde{Y_\infty})^{\Lambda} \rightarrow C^*(\widetilde{\overline{X}} \subset \widetilde{X})^\Gamma} \rightarrow C_{(\psi,\ev_{\infty})}$$ gives rise to isomorphisms of $K$-theory groups. Analogous statements hold when $C^*$ is replaced by $C^*_{L}$ and $C^*_{L,0}$.
\end{proposition}
\begin{proof}
Note that the mapping cone of the identity map on $ C^*(\widetilde{Y} \times \reals)^{\Lambda\times\reals}$ is contractible and thus has trivial $K$-theory. The statement then follows from the long exact sequence of $K$-theory groups associated to the short exact sequence of mapping cones
$$0\rightarrow C_{C^*(\widetilde{Y} \subset \widetilde{Y_\infty})^{\Lambda} \rightarrow C^*(\widetilde{\overline{X}} \subset \widetilde{X})^\Gamma} \rightarrow C_{(\psi,\ev_\infty)}\rightarrow C_{\iden} \rightarrow 0$$.
\end{proof}
\begin{remark}
	Recall that associated to a homomorphism $f:A\rightarrow B$ of $C^\ast$-algebras there is a mapping cone short exact sequence
	$$0 \rightarrow SB \rightarrow C_f \rightarrow A \rightarrow 0$$
	where $C_f$ denotes the mapping cone of $f$ and $S$ denotes the suspension. Up to applying Bott periodicity the inclusion $SB \rightarrow C_f$ gives rise to group homomorphisms $K_*(B) \rightarrow K_*(SC_f)$.
	In the following we apply these observations to the homomorphisms $\beta$, $\beta_L$ and $\beta_{L,0}$ (see \ref{eq:beta}) and the maps $(\psi,{\ev_\infty})$, $(\psi_{L},{\ev_\infty}_{L})$ and $(\psi_{L,0},{\ev_\infty}_{L,0})$ introduced in Remark \ref{rem:phievinfty}.
\end{remark}
\begin{remark}\label{rem:relabs}
	Using the inclusions $\widetilde{Y} \rightarrow \widetilde{Y}_\infty$ and $\widetilde{\overline{X}} \rightarrow \widetilde{X}$ and the Propositions \ref{prop:relabs} and \ref{prop:relabsl} we obtain isomorphisms
	$$K_*(C_{\beta_{(L,(0))}}) \xrightarrow{\cong} K_*(C_{C_{(L,(0))}^*(\widetilde{Y} \subset \widetilde{Y}_\infty)^{\Lambda} \rightarrow C_{(L,(0))}^*(\widetilde{\overline{X}} \subset \widetilde{X})^{\Gamma}}).$$
\end{remark}
\begin{proposition}
There is a commutative diagram of long exact sequences
$$\begin{tikzcd}
   K_*(C^*_{L,0}(\widetilde{X})^{\Gamma,\reals_+,\Lambda}) \arrow{r} \arrow{d} & K_*(C^*_{L}(\widetilde{X})^{\Gamma,\reals_+,\Lambda}) \arrow{r} \arrow{d} & K_*(C^*(\widetilde{X})^{\Gamma,\reals_+,\Lambda}) \arrow{d}\\
   K_*(SC_{\beta_{L,0}}) \arrow{r} & K_*(SC_{\beta_L}) \arrow{r} & K_*(SC_{\beta}),
\end{tikzcd}$$
where the vertical maps are given by the compositions
$$K_*(C^*_{(L,(0))}(\widetilde{X})^{\Gamma,\mathbb{R}_+,\Lambda}) \rightarrow
 K_*(SC_{(\psi_{(L,(0))},{\ev_\infty}_{(L,(0))})})$$ 
 $$\cong
 K_*(SC_{C_{(L,(0))}^*(\widetilde{Y} \subset \widetilde{Y}_\infty)^{\Lambda} \rightarrow C_{(L,(0))}^*(\widetilde{\overline{X}} \subset \widetilde{X})^{\Gamma}})
 \cong K_*(SC_{\beta_{(L,(0))}}).$$
 See Remark \ref{rem:relabs} for the definition of the last isomorphism.
 \label{prop:maptoCWY}
\end{proposition}
\begin{proof}
The diagram in the claim of the proposition is obtained by composing the diagrams
$$\begin{tikzcd}
   K_*(C^*_{L,0}(\widetilde{X})^{\Gamma,\reals_+,\Lambda}) \arrow{r} \arrow{d} & K_*(C^*_{L}(\widetilde{X})^{\Gamma,\reals_+,\Lambda}) \arrow{r} \arrow{d} & K_*(C^*(\widetilde{X})^{\Gamma,\reals_+,\Lambda}) \arrow{d}\\
   K_*(SC_{(\psi_{L,0},{\ev_\infty}_{L,0})}) \arrow{r} & K_*(SC_{(\psi_{L},{\ev_\infty}_{L})}) \arrow{r} & K_*(SC_{(\psi,{\ev_\infty})}),
\end{tikzcd}$$
and
$$\begin{tikzcd}
   K_*(SC_{(\psi_{L,0},{\ev_\infty}_{L,0})}) \arrow{r} & K_*(SC_{(\psi_{L},{\ev_\infty}_{L})}) \arrow{r} & K_*(SC_{(\psi,{\ev_\infty})}) \\
   K_*(SC_{\beta_{L,0}}) \arrow{r} \arrow{u}{\cong} & K_*(SC_{\beta_L}) \arrow{r} \arrow{u}{\cong} & K_*(SC_{\beta}) \arrow{u}{\cong},
\end{tikzcd}$$
where $(\psi_{(L,(0))},{\ev_\infty}_{(L,(0))})$ denotes the map $C_{(L,(0))}^*(\widetilde{Y_\infty})^{\Lambda,\mathbb{R}_+} \rightarrow C_{(L,(0))}^*(\widetilde{X})^{\Gamma,\mathbb{R}_+,\Lambda}$ introduced in Remark \ref{rem:phievinfty}. The commutativity of the first diagram is due to the naturality of the mapping cone exact sequence and the commutativity of the second diagram is clear.
\end{proof}
Denote the image of the fundamental class of the Dirac operator on $\widetilde{X}$ under the composition
$$K_*(C^*_L(\widetilde{X})^\Gamma) \rightarrow K_*(SC_{C_L^*(\widetilde{Y_\infty})^{\Lambda} \rightarrow C_L^*(\widetilde{X})^{\Gamma}}) \cong K_*(SC_{C_L^*(\widetilde{Y} \subset \widetilde{Y_\infty})^{\Lambda} \rightarrow C_L^*(\widetilde{\overline{X}} \subset \widetilde{X})^{\Gamma}})$$ by $[D_{\widetilde{\overline{X}},\widetilde{Y}}]$.
\begin{lemma}
The class $[D_{\widetilde{X},\widetilde{Y}}]$ maps to $[D_{\widetilde{\overline{X}},\widetilde{Y}}]$ under the map $K_*(C^*_{L}(\widetilde{X})^{\Gamma,\reals_+,\Lambda}) \rightarrow K_*(SC_{\psi_L})$ of Proposition \ref{prop:maptoCWY}.
\label{lemma:comparefundclasses}
\end{lemma}
\begin{proof}
We first note that the commutativity of the diagram
$$\begin{tikzcd}
   K_*(C^*_L(\widetilde{Y_\infty})^{\Lambda,\reals_+}) \arrow{r} \arrow{d} & K_*(C^*_L(\widetilde{X})^{\Gamma,\reals_+,\Lambda}) \arrow{d}\\
   K_*(C^*_L(\widetilde{Y_\infty})^\Lambda) \arrow{r} & K_*(C^*_L(\widetilde{X})^\Gamma),
\end{tikzcd}$$
where the second vertical map is given by the composition of the projection onto the $C^*_L(\widetilde{X})^{\Gamma,\reals_+}$ component followed by the inclusion, implies that of 
$$\begin{tikzcd}
   K_*(C^*_L(\widetilde{X})^{\Gamma,\reals_+,\Lambda}) \arrow{r} \arrow{d} & K_*(SC_{C_L^*(\widetilde{Y_\infty})^{\Lambda,\mathbb{R}_+} \rightarrow C_L^*(\widetilde{X})^{\Gamma,\mathbb{R}_+,\Lambda}}) \arrow{d}\\
   K_*(C^*_L(\widetilde{X})^\Gamma) \arrow{r} & K_*(SC_{C_L^*(\widetilde{Y_\infty})^\Lambda \rightarrow C_L^*(\widetilde{X})^\Gamma}).
\end{tikzcd}$$
Furthermore, the diagram
$$\begin{tikzcd}
K_*(SC_{C_L^*(\widetilde{Y} \subset \widetilde{Y_\infty})^{\Lambda} \rightarrow C_L^*(\widetilde{\overline{X}} \subset \widetilde{X})^{\Gamma}}) \arrow{rd} \arrow{r} & K_*(SC_{C_L^*(\widetilde{Y_\infty})^{\Lambda,\mathbb{R}_+} \rightarrow C_L^*(\widetilde{X})^{\Gamma,\mathbb{R}_+,\Lambda}}) \arrow{d}\\
   & K_*(SC_{C_L^*(\widetilde{Y_\infty})^\Lambda \rightarrow C_L^*(\widetilde{X})^\Gamma}),
\end{tikzcd}$$
where all the arrows are isomorphisms, is commutative. The claim then follows from the commutativity of the latter two diagrams and the fact that $[D_{\widetilde{X},\widetilde{Y}}]$ lifts the fundamental class of $\widetilde{X}$
\end{proof}
\begin{corollary}
The $(\Gamma,\Lambda)$-index of the Dirac operator associated to $(X,Y,\iota)$ maps to the relative index of the Dirac operator on $\overline{X}$ under $K_*(C^*(\widetilde{X})^{\Gamma,\reals_+,\Lambda}) \rightarrow K_*(SC_{\psi})$ defined in Proposition \ref{prop:maptoCWY}.
\end{corollary}
Combining Lemma \ref{lemma:comparefundclasses} and Proposition \ref{prop:vanishing} gives a new (and very natural) proof of the following
\begin{proposition}
The nonvanishing of the relative index of the Dirac operator on a manifold with boundary is an obstruction to the existence of a positive scalar metric which is collared at the boundary.
\end{proposition}

\subsection{Localised Indices and the Relative Index}
Given a metric $g$ on $X$ which has positive scalar curvature outside $\overline{X}$, one can define a localised coarse index in $K_n(C^*(\widetilde{\overline{X}})^\Gamma)$.
In \cite{SS18} it was shown that this index maps to the relative index of $\overline{X}$. We quickly recall the construction of the localised index and use the machinery developed previously to give a short proof of the latter statement.
\begin{definition}
Denote by $C^*_{L,\widetilde{\overline{X}}}(\widetilde{X})^\Gamma$ the preimage of $C^*(\widetilde{\overline{X}} \subset \widetilde{X})^\Gamma$ under $\ev_1 : C^*_L(\widetilde{X})^\Gamma \rightarrow C^*(\widetilde{X})^\Gamma$.
\end{definition}
Suppose that the scalar curvature of the metric restricted to the complement of $\overline{X}$ is bounded from below by $\epsilon > 0$. The following proposition is well-known. As in \cite{RZA} one can define a partial $\rho$-invariant $\rho^\Gamma_{\overline{X}}(g) \in K_n(C^*_{L,\widetilde{\overline{X}}}(\widetilde{X})^\Gamma)$ using the morphism 
$$\varphi_{D_{\widetilde{X}}} \circ \psi: \mathcal{S} \rightarrow C^*_{L,\widetilde{\overline{X}}}(\widetilde{X};\Cl_n)^\Gamma.$$
\begin{definition}
The localised coarse index $\ind_{\widetilde{\overline{X}}}^\Gamma(g)$ is the image of $\rho^\Gamma_{\overline{X}}(g)$ under $(\ev_1)_*: K_n(C^*_{L,\widetilde{\overline{X}}}(\widetilde{X})^\Gamma) \rightarrow K_n(C^*(\widetilde{\overline{X}} \subset \widetilde{X})^\Gamma)$.
\end{definition}
\begin{remark}
Note that in the above situation we can also define $\rho^{\Gamma,\Lambda}_{\overline{X}}(g)$. Furthermore, we note that the commutativity of the diagram
$$\begin{tikzcd}
K_*(C^*_{L,\widetilde{\overline{X}}}(\widetilde{X})^{\Gamma,\reals_+,\Lambda}) \arrow{rd} \arrow{r} & K_*(C^*_{L,\widetilde{\overline{X}}}(\widetilde{X})^\Gamma) \arrow{d}{{(\ev_1)}_*}\\
   & K_*(C^*(\widetilde{\overline{X}} \subset \widetilde{X})^\Gamma),
\end{tikzcd}$$ and the fact that $\rho^{\Gamma,\Lambda}_{\overline{X}}(g)$ is a lift of $\rho^{\Gamma}_{\overline{X}}(g)$ under the horizontal map imply that $\ind_{\widetilde{\overline{X}}}^\Gamma(g)$ is the image of $\rho^{\Gamma,\Lambda}_{\overline{X}}(g)$ under the map $K_*(C^*_{L,\widetilde{\overline{X}}}(\widetilde{X})^{\Gamma,\reals_+,\Lambda}) \rightarrow K_*(C^*(\widetilde{\overline{X}} \subset \widetilde{X})^\Gamma)$.
\end{remark}
The following lemma is a simple observation
\begin{lemma}
The following diagram is commutative
$$\begin{tikzcd}
   K_*(C^*_{L,\widetilde{\overline{X}}}(\widetilde{X})^{\Gamma,\reals_+,\Lambda})  \arrow{r} \arrow{d} & K_*(C^*(\widetilde{\overline{X}} \subset \widetilde{X})^\Gamma) \arrow{r} \arrow{d} & K_*(SC_{C^*(\widetilde{Y} \subset \widetilde{Y_\infty})^{\Lambda} \rightarrow C^*(\widetilde{\overline{X}} \subset \widetilde{X})^\Gamma}) \arrow{d}\\
   K_*(C^*_L(\widetilde{X})^{\Gamma,\reals_+,\Lambda}) \arrow{r} & K_*(C^*(\widetilde{X})^{\Gamma,\reals_+,\Lambda}) \arrow{r} & K_*(SC_{C^*(\widetilde{Y_\infty})^{\Lambda,\mathbb{R}_+} \rightarrow C^*(\widetilde{X})^{\Gamma,\mathbb{R}_+,\Lambda}}).
\end{tikzcd}$$
\end{lemma}
Suppose $\overline{X}$ is compact. Then $K_*(C^*(\widetilde{\overline{X}} \subset \widetilde{X})^\Gamma) \cong K_*(C^*(\Gamma))$. Using the previous remark and lemma we obtain the following corollary, which was one of the main statements of \cite{SS18}.
\begin{corollary}
Suppose $\overline{X}$ is compact. Then $\ind_{\widetilde{\overline{X}}}^\Gamma(g)$ maps to the relative index of Chang, Weinberger and Yu under the map $K_*(C^*(\Gamma)) \rightarrow K_*(C^*(\Gamma,\Lambda))$.
\end{corollary}
\bibliography{references}
\end{document}